\documentclass[pdflatex,sn-mathphys-num]{sn-jnl}

\usepackage{graphicx}%
\usepackage{multirow}%
\usepackage{amsmath,amssymb,amsfonts}%
\usepackage{amsthm}%
\usepackage{mathrsfs}%
\usepackage[title]{appendix}%
\usepackage{xcolor}%
\usepackage{textcomp}%
\usepackage{manyfoot}%
\usepackage{booktabs}%
\usepackage{algorithm}%
\usepackage{algorithmicx}%
\usepackage{algpseudocode}%
\usepackage{listings}%
\usepackage{subcaption}
\usepackage{multirow}
\usepackage[shortlabels]{enumitem}

\usepackage{pifont}

\usepackage{tikz}
\usetikzlibrary{arrows,arrows.meta,intersections,shapes.geometric,positioning,calc,patterns}

\usepackage{resizegather}

\theoremstyle{thmstyleone}%
\newtheorem{theorem}{Theorem}[section]
\newtheorem{proposition}[theorem]{Proposition}%

\theoremstyle{thmstyletwo}%
\newtheorem{remark}{Remark}%

\theoremstyle{definition}%

\DeclareMathOperator*{\argmin}{\arg\min}

\begin{document}

\title[Effectively Leveraging Momentum Terms in Stochastic Line Search Frameworks for Fast Optimization of Finite-Sum Problems]{Effectively Leveraging Momentum Terms in Stochastic Line Search Frameworks for Fast Optimization of Finite-Sum Problems}


\author[1]{\fnm{Matteo} \sur{Lapucci}}\email{matteo.lapucci@unifi.it}

\author*[1]{\fnm{Davide} \sur{Pucci}}\email{davide.pucci@unifi.it}

\affil[1]{\orgdiv{Global Optimization Laboratory, Department of Information Engineering}, \orgname{University of Florence}, \orgaddress{\street{Via di Santa Marta, 3}, \city{Florence}, \postcode{50139}, \country{Italy}}}


\abstract{
	In this work, we address unconstrained finite-sum optimization problems, with particular focus on instances originating in large scale deep learning scenarios. Our main interest lies in the exploration of the relationship between recent line search approaches for stochastic optimization in the overparametrized regime and momentum directions. First, we point out that combining these two elements with computational benefits is not straightforward. To this aim, we propose a solution based on mini-batch persistency. We then introduce an algorithmic framework that exploits a mix of data persistency, conjugate-gradient type rules for the definition of the momentum parameter and stochastic line searches. The resulting algorithm provably possesses convergence properties under suitable assumptions and is empirically shown to outperform other popular methods from the literature, obtaining state-of-the-art results in both convex and nonconvex large scale training problems.}

\keywords{Finite-sum optimization, Stochastic line search, Momentum, Data persistency}

\pacs[MSC Classification]{90C30, 90C26, 90C06, 65K05, 68T07}

\maketitle

\section{Introduction}
\label{sec:intro}
In this paper we focus on finite-sum minimization problems where the objective function has a large number of possibly nonconvex terms. Precisely, we address the problem
\begin{equation} \label{eq:fin_sum_problem}
	\min_{x \in \mathbb{R}^n} f(x) = \frac{1}{N} \sum_{i=1}^{N} f_i(x),
\end{equation}
where $N$ is large and $f_i: \mathbb{R}^n \rightarrow \mathbb{R}$ are differentiable, possibly nonconvex functions for all $i \in \{1, \dots, N\}$. The development of high-performance algorithms for this scenario has been driven in recent years by the explosion of deep-learning applications: indeed, supervised learning tasks are instances of problem \eqref{eq:fin_sum_problem}. 

State-of-the-art methods to tackle problems of this form are primarily based on stochastic gradient descent (SGD) \cite{robbins1951} and its accelerated variants \cite{polyak1964, nesterov1983}. The use of stochastic gradient methods allows to leverage data redundancy to approach good solutions while carrying out low-cost iterations. A thorough discussion on stochastic gradient methods and their application in data science is available in \cite{bottou2018}.

Adaptive SGD approaches \cite{duchi2011, tieleman2012, zeiler2012, kingma2015} have established through the years as solid choices in the considered scenario. In particular, Adam algorithm \cite{kingma2015} is now largely considered the preferred optimizer in deep learning settings. 
In an alternative stream of research, variance-reduced approaches \cite{schmidt2017, rie2013, gower2023} have been studied and shown to possess faster convergence rates than the base algorithm and in line with full-batch gradient descent. However, in practice, these methods require significant computational resources, either needing periodic full gradient computations or high memory usage; for this reason, variance reduction methods are often found to be impractical to use in real-world scenarios with deep architectures and very large datasets. Moreover, the assumptions made to obtain the nice convergence results have been shown not to be satisfied in deep learning scenarios. In fact, no improvement in the convergence rate \cite{defazio2019} w.r.t.\ SGD is obtainable under more realistic assumptions.

In more recent years, the focus turned to the analysis of stochastic gradient methods under a specific assumption that is reasonable with modern deep learning models: the interpolation regime. In this scenario, the convergence rate of SGD was finally shown to match that of full-batch methods \cite{ma2018, vaswani2020}, giving some justification to the observed behavior of these methods. 
Moreover, under interpolation assumptions it was possible to design well-grounded strategies to choose the step size in an adaptive manner \cite{mutschler2020, loizou2021}. 

Now, one of the most promising developments concerns the use of stochastic line searches. In particular, the classical Armijo rule was adapted to be successfully used in the incremental scenario \cite{vaswani2019}. More recently, a nonmonotone version of the stochastic Armijo line search was proposed \cite{galli2023}. With these techniques, not only the convergence rate of standard gradient descent is shown to be obtainable, but very encouraging numerical results were also observed.

As widely recognized in the machine learning literature, the addition of momentum terms in the update rule has been known to be beneficial within SGD approaches \cite{sutskever2013importance, sebbouh2021almost, tseng1998incremental, jelassi2022towards, gitman2019understanding}; typically, this is attributed to the stabilizing effect on the direction and the speed-up in low curvature regions. It is still not clear, however, how this kind of additional term, which modifies the search direction, could be suitably handled in line search based frameworks. A heuristic choice for the step size in presence of momentum terms under interpolation assumptions has been proposed by \cite{wang2023}, but no decrease condition is coupled with this strategy. On the other hand, a backtracking strategy has been recently suggested, ensuring the resulting direction is of descent for the considered mini-batch \cite{fan2023}; for this latter approach, however, we will outline massive shortcomings both on the theoretical and on the computational sides.

In this manuscript, we therefore deal with the above challenge. 
Firstly, we shed light on an intrinsic issue related to the momentum direction in the incremental regime; in order to overcome it, at least partially, we propose to exploit the concept of mini-batch persistency \cite{fischetti2018}, which we also observe being beneficial on its own in certain settings. 

We can then focus on a suitable rule to define the parameter $\beta$ associated with the momentum term in the direction update. The proposed method exploits the strong connection of momentum methods and classical nonlinear conjugate gradient algorithms \cite[Ch.\ 12]{grippo2023introduction}: in both cases, the search direction is a linear combination of the current gradients and the previous direction, and the border between the two families of approaches is thin. The update rule is combined with safeguarding strategies based on restarting \cite{chan2022nonlinear} or subspace optimization \cite{lapucci2024globally}, that allow to guarantee crucial assumptions for the soundness of the resulting algorithm \cite{lapucci2024convergenceconditionsstochasticline}.

For the proposed algorithmic framework, we provide a brief discussion concerning the major, substantial challenges associated with the theoretical analysis of the proposed method - namely, the biasedness of the search direction in presence of data persistency and the relationship of the direction itself with the true gradient - and present a set of workarounds that make the algorithm provably convergent under PL condition and interpolation.

We finally present the results of in-depth computational experiments, showing that it is competitive and even outperforms in certain scenarios the main state-of-the-art optimizers for both convex (linear models) and nonconvex (deep networks) learning tasks.

The reminder of this paper is organized as follows: in Section \ref{sec:prelim} we recall the main concepts and algorithms to tackle finite sum problems, with a particular focus on algorithms that employ stochastic line searches under interpolation assumptions. In Section \ref{sec:persist}, we discuss the difficulties encountered when momentum terms are employed in combination with line searches (Section \ref{sec:sls_mom}), proposing in Section \ref{sec:mb_pers} a strategy to suitably leverage momentum information. In Section \ref{sec:cg} we describe some conjugate gradient rules to determine a suitable value for the parameter $\beta$ associated to the momentum term, that exploit data persistency. In Section \ref{sec:alg}, we present our algorithmic framework, which exploits momentum terms alongside stochastic line searches. In Sections \ref{sec:bias} and \ref{sec:direction} we address the main challenges involved in establishing convergence results for the proposed algorithmic framework, then reporting the main theoretical result in Section \ref{sec:main_theorem}. Section \ref{sec:exp} is devoted to computational experiments aimed at assessing the effectiveness of the proposed method. Finally, Section \ref{sec:conc} provides concluding remarks. 

\section{Preliminaries}
\label{sec:prelim}
Problem \eqref{eq:fin_sum_problem} is typically solved by means of stochastic gradient descent (SGD) type algorithms. Since $N$ is very large and the evaluation of exact derivatives is therefore expensive, at each iteration a cheap approximation of the gradient $\nabla f(x)$ is considered. The approximation is usually obtained randomly sampling a subset (referred to as \textit{mini-batch}) $B_k\subset\{1,\ldots,N\}$ and thus defining 
\begin{equation}
	\label{eq:gen_minibatch_fg}
	f_k(x) = \frac{1}{|B_k|}\sum_{i \in {B_k}}f_i(x),\qquad g_k(x)=\nabla f_k(x) = \frac{1}{|B_k|}\sum_{i \in {B_k}}\nabla f_i(x).
\end{equation}
The update rule is then given by $x^{k+1}=x^k+\alpha_kd_k$, where the step size $\alpha_k$ is often referred to as \textit{learning rate} and $d_k$ is the search direction.
A deep presentation of stochastic gradient methods for machine learning tasks is available in \cite{bottou2018}.
Under regularity assumptions, and using a suitable decreasing sequence of steps $\{\alpha_k\}$, the iterative scheme allows to guarantee convergence to stationarity in expectation. The associated worst-case complexity bound is $\mathcal{O}(\frac{1}{\epsilon^4})$ in the nonconvex setting \cite{ghadimi2013}, which becomes $\mathcal{O}(\frac{1}{\epsilon^2})$ and $\mathcal{O}(\frac{1}{\epsilon})$ under convexity and strong convexity assumptions, respectively \cite{bottou2018}. These bounds are worse than those obtainable by full-batch methods, which should be thus superior. Yet, the empirical evidence collected training deep machine learning models tells another story. 

Recently, some explanation to this gap between theory and numerical observations was found: modern machine learning models are usually expressive enough to fit any data point in the training set \cite{liang2020,ma2018}, i.e., the so-called \textit{interpolation condition} holds, which means that, given $x^* \in \arg \min_{x \in \mathbb{R}^n} f(x)$, then $x^* \in \arg \min_{x \in \mathbb{R}^n} f_i(x)$ for all $i \in \{1,\dots, N\}$.  
When the interpolation property holds, SGD methods can be shown to match gradient descent complexity of $\mathcal{O}(\frac{1}{\epsilon})$ in the convex case and its linear rate under strong convexity. 

Additionally, the landscape of the training loss of deep models has been studied \cite{liu2022}; while convexity and strong convexity seem not to hold, not even locally, another property appears to hold in large portions of the space: the \textit{Polyak-Lojasiewicz} (PL) condition \cite{polyak1987introduction}. Very interestingly, the linear convergence rate was shown to hold for SGD even with this weaker property, if interpolation is taken into account \cite{vaswani2019, loizou2021, galli2023}.

Interpolation implies a tighter relationship between the stochastic functions $f_{k}(x)$ and the true loss $f(x)$; this strong bond allowed to define more sophisticated strategies to choose the step size, with nice results both from the theoretical and the computational sides. Hereafter, we summarize some of these relevant linearly convergent approaches. 

Loizou et al.\ \cite{loizou2021} proposed a stochastic variant of the Polyak step size \cite{polyak1987introduction} for guessing the most promising step size for SGD at each iteration.
In particular, the \textit{Stochastic Polyak step size} (SPS) is obtained assuming the optimal function value for the current mini-batch objective $f_k$ to be some known value $f_{k}^*$ and reachable moving along $d_k$; the SPS is consequently obtained, minimizing a first-order model, setting
\begin{equation}
	\label{eq:sps}
	\alpha^{\text{SPS}}_k = \frac{f_{k}(x^k) - f_{k}^*}{c ||\nabla f_{k}(x^k)||^2},
\end{equation}
where $c > 0$ is a normalization parameter. In the same work, the bounded variant $\text{SPS}_{\text{max}}$ is also presented, considering $\alpha_k = \min \left\{\alpha^{\text{SPS}}_k, \alpha_\text{max} \right\}$ where $\alpha_\text{max} > 0$ is a suitable upper bound for the step size. The assumption on the value of $f_{k}^*$ is reasonable in most cases, as it is linked to the interpolation condition and the non-negative nature of loss functions in most machine learning applications. For instance, it can be set to 0 with unregularized least squares, or logistic losses or by some closed-form rule in the corresponding regularized cases (see, e.g., \cite[Sec.\ 2.1]{loizou2021}).

Inspired by the arguably top performing methods in classical nonlinear optimization, line searches have also been considered to achieve fast convergence in the interpolation regime. 
The first stochastic line search specifically designed for this scenario was introduced in \cite{vaswani2019}: the extension of an Armijo-type condition \cite{armijo1966} to the stochastic case is shown to lead to the desired linear convergence rate under reasonable assumptions. The line search is structured so that the progress in the current stochastic function is measured, looking for a sufficient decrease that depends on the stochastic gradient direction; 
formally, the proposed Armijo condition is
\begin{equation}
	\label{eq:stoch_armijo}
	f_{k}(x^{k}-\alpha_k g_k(x^k)) \leq f_k(x^k)-\gamma\alpha_k\| g_k(x^k)\|^2.
\end{equation}
Checking whether this condition is satisfied or not for a given step size $\alpha$ is rather cheap, only needing to make an additional evaluation of the stochastic function; a suitable step size $\alpha_k$ can thus be easily computed by a classical backtracking procedure; possibly, we would like to select a reasonable initial step size to reduce the number of backtracks, and thus function evaluations, to the bone; for this reason, heuristic strategies have been devised to determine a good guess step.

In fact, the strict decrease on the current stochastic objective might appear as an unnecessarily strong  condition to impose, especially in highly nonlinear and nonconvex problems. This observation was the main motivation, in \cite{galli2023}, to introduce a nonmonotone variant of the Armijo type condition. 
Letting $$C_k = \max\{\tilde{C}_k,f_{k}(x^k)\}, \quad \tilde{C}_k = \frac{\xi Q_kC_{k-1}+f_k(x^k)}{Q_{k+1}},\quad Q_{k+1} = \xi Q_k + 1,$$
so that $C_k$ is a value greater or equal to both a linear combination of previously computed (stochastic) function values the and the current value $f_k(x^k)$, the nonmonotone decrease condition is given by
\begin{equation}
	\label{eq:nonmonotone}
	f_{k}(x^{k}-\alpha_kg_k(x^k)) \leq C_k-\gamma\alpha_k\|g_k(x^k)\|^2.
\end{equation}
The employment of this condition instead of \eqref{eq:stoch_armijo} in line search based mini-batch GD might contribute to reduce the function evaluations required by repeated backtracks and to more often accept aggressive step sizes, without sacrificing convergence speed. Combined with an initial guess step of the form \eqref{eq:sps}, the resulting algorithm, named PoNoS, proves to be computationally very efficient at solving nonconvex learning tasks.

\section{Integrating Momentum Terms with SLS}
\label{sec:persist}
It is broadly accepted that the employment of momentum terms in gradient based optimization of finite-sum functions has a considerable positive impact \cite{bottou2018,wright2022optimization}. Specifically, heavy-ball like updates follow rules of the form
\begin{equation}
	\label{eq:momentum_update}
	x^{k+1} = x^k - \alpha_k g_k(x^k) + \beta_k(x^k-x^{k-1}),
\end{equation}
where $\beta_k$ is a suitable positive scalar. In the stochastic setting, partly repeating the previous update mainly helps at stabilizing the optimization trajectory and at taking larger steps in low-curvature regions. Momentum is thus employed in most state-of-the-art optimizers for large-scale nonconvex finite-sum problems, e.g., it is included in Adam \cite{kingma2015}.

It should not thus be surprising that, within the most recent lines of research concerning the selection of the step size, momentum direction ended up being considered. In this Section we discuss existing strategies to integrate momentum terms in modern stochastic optimization frameworks, highlighting a major shortcoming. We then propose an approach to suitably overcome the issue.

\subsection{Existing Strategies and Computational Shortcomings}
\label{sec:sls_mom}
A generalization of the Stochastic Polyak step size was introduced to determine the appropriate step size when considering a momentum-type direction \cite{wang2023}. In particular, the step size associated with moving average of gradients (MAG) directions of the form $d_k = -(1-\beta) \nabla f(x^k) + \beta (x^k - x^{k-1})$, being $\beta \in (0, 1)$ a fixed parameter, is defined as 
\begin{equation}
	\label{eq:sps_momentum}
	\alpha_k = \min \left\{\frac{f_k(x^k) - f_k^*}{c ||d_k||^2}, \alpha_\text{max} \right\},
\end{equation}
resulting in a simple, yet powerful extension of \eqref{eq:sps}. Of course, an analogous rule can be applied with directions of the form 
\begin{equation}
	\label{eq:hb_dir}
	d_k = -g_k(x^k) + \beta_k(x^{k}-x^{k-1});
\end{equation} 
yet, no hint is provided for the definition of a suitable $\beta_k$.

The combination of a momentum-type direction and the use of a stochastic line search was recently discussed \cite{fan2023}. The resulting algorithm, (called MSL\_SGDM) revolves around directions of the form \eqref{eq:hb_dir}. In case, for a given value of $\beta$, $d_k$ is not a descent direction for the current stochastic function $f_k$ (i.e., $d_k^T \nabla f_{_k} (x^k) > 0$), it is modified by either setting $\beta=0$ or by reducing $\beta$ until $d_k^T\nabla f_{k} (x^k) < 0$. 
Once $d_k$ is ensured to be of descent, the step size $\alpha_k$ in the update 
\begin{equation}
	\label{eq:gen_update rule}
	x^{k+1} = x^k+\alpha_kd_k
\end{equation} 
can be computed according to satisfy a generalized Armijo condition according to 
\begin{equation}
	\label{eq:gen_sls_rule}
	\alpha_k= \max_{j=0,1,\ldots}\{\alpha_0^k\delta^j\mid f_{k}(x_k+\alpha_0^k\delta^j d_k)\le f_{k}(x_k) +\gamma \alpha_0^k\delta^j d_k^T\nabla f_k(x_k) \}.
\end{equation}

By means of the same strategy, other families of directions could be considered, as long as $d_k$ can be manipulated to satisfy the ``stochastic'' descent property for the mini-batch; Adam in particular is considered in \cite{fan2023}, but Mean Average Gradient (MAG)  or Nesterov-type directions could also be handled by this framework in principle.

\medskip
Although apparently well planned, the strategy from \cite{fan2023} hides a significant drawback, as we will argue hereafter. Surely, it is a given fact that, in the incremental setup, using a momentum-based algorithm somehow allows to mitigate the presence of noise in gradient evaluations caused by the usage of mini-batches. 

Yet, if we carry out a line search to determine the step size at some iteration $k$, we are basing our choice on the behavior of the stochastic loss $f_{k}$ (depending in practice on the randomly drawn samples in mini-batch) along direction $d_k$. The term $x^k - x^{k-1}$ being added to the stochastic gradient, on the other hand, comes from a variables update that led a decrease in $f_{k-1}$; if $f_{k}$ and $f_{{k-1}}$ substantially differ from each other, the momentum direction $x^k - x^{k-1}$ might be unrelated to the new mini-batch loss, and possibly even point to increasing values of $f_{k}$. It might thus be not so unusual for the value of $\beta$ to be shrunk to very small values to obtain the decent property, and then several backtracks might also be needed to obtain sufficiently small step sizes guaranteeing the sufficient decrease.    
In order to overcome this issue, we would need $f_{{k-1}}$ to be somewhat similar to $f_{k}$.

\subsection{Leveraging Mini-Batch Persistency}
\label{sec:mb_pers}
For the problem highlighted in the previous Section \ref{sec:sls_mom}, we propose a practical solution that can be employed when the stochastic functions $f_k$ and $g_k$ are of the form \eqref{eq:gen_minibatch_fg}.

The strategy to enforce some form of similarity between $f_k$ and $f_{k-1}$ consists in the selection of some indexes from the preceding mini-batch ${B}_{k-1}$ when the new mini-batch ${B}_k$ is formed, so that $$\mathcal{R}_{k-1} = {B}_{k-1} \cap {B}_k \neq \emptyset.$$
In the context of machine learning, this idea translates into using a shared portion of data samples in successive mini-batches. The larger the overlap between mini-batches, the more similar $f_{{k-1}}$ and $f_{k}$ will be. This strategy is referred to as \textit{mini-batch persistency}, and it was observed to be beneficial even with vanilla SGD algorithms, especially when large batch sizes are considered \cite{fischetti2018}.

This idea will be crucial for the rest of the work. We therefore believe it is important to provide here readers with some additional arguments, of numerical nature. 
We thus present the results of preliminary computational experiments aimed at assessing the effect of mini-batch persistency in mini-batch GD optimization.
We considered two learning tasks: 
\begin{itemize}
	\item a convex problem - an RBF-kernel classifier  to be trained  on the \texttt{ijcnn} dataset\footnote{available at \href{https://www.csie.ntu.edu.tw/~cjlin/libsvmtools/datasets/binary.html}{csie.ntu.edu.tw/~cjlin/libsvmtools/datasets}} from LIBSVM \cite{libsvm2011};
	\item a nonconvex problem - a multi-layer perceptron with 1000 hidden units and ReLU activations to be trained on the \texttt{MNIST} \cite{mnist1998} dataset\footnote{available at \href{http://yann.lecun.com/exdb/mnist/}{yann.lecun.com/exdb/mnist/}}. 
\end{itemize} 

Both problems have been solved by means of vanilla mini-batch GD: for the two problems we set a fixed learning rate of $10^3$ and $10^{-2}$, respectively. 

During the execution of the algorithm, at the beginning of each iteration  we measure the angle between the momentum term $x^{k}-x^{k-1}$ and the (negative) stochastic gradient $-\nabla f_{k} (x^k)$ that we would get using different mini-batches $B_k$ obtained for various levels of overlap. In particular, we consider mini-batches with 0\%, 25\%, 50\%, 75\% and 100\% overlap (the latter showing what would happen if optimization were kept going on the same mini-batch). Note that the actual update will always be done with no momentum and no mini-batch overlap.  

The test was repeated for batch sizes 128 and 512, to observe the effect of persistency with small and large batch sizes.

Figure \ref{fig:overlap_analysis} shows the results of this analysis. We can clearly observe that the greater is the overlap, the closer the momentum direction is to $-\nabla f_{{k}} (x^k)$. We thus deduce that, unsurprisingly, with overlapped mini-batches functions $f_{{k-1}}$ and $f_{{k}}$ end up being more similar; this way, larger values of $\beta$ will be valid to obtain an overall heavy-ball direction of descent for $f_{{k}}$. In Table \ref{tab:overlap_analysis} we further report additional information on the overall number of times $x^{k}-x^{k-1}$ resulted in a non-descent direction for $f_{k}$ at $x_k$; the conclusions we can draw are similar as above.

\begin{figure}[htbp]
	\centering
	\subfloat[Kernel classifier on \texttt{ijcnn} - batch size 128.]{\includegraphics[width=.45\textwidth]{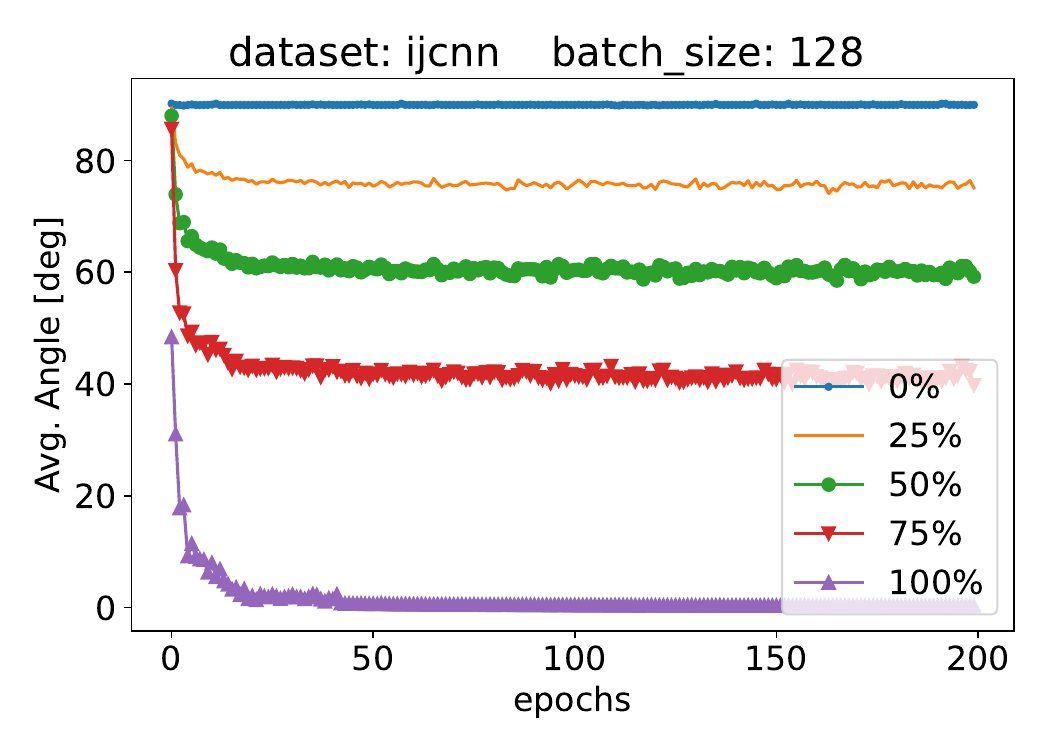}}
	\hfil
	\subfloat[Kernel classifier on \texttt{ijcnn} - batch size 512.]{\includegraphics[width=.45\textwidth]{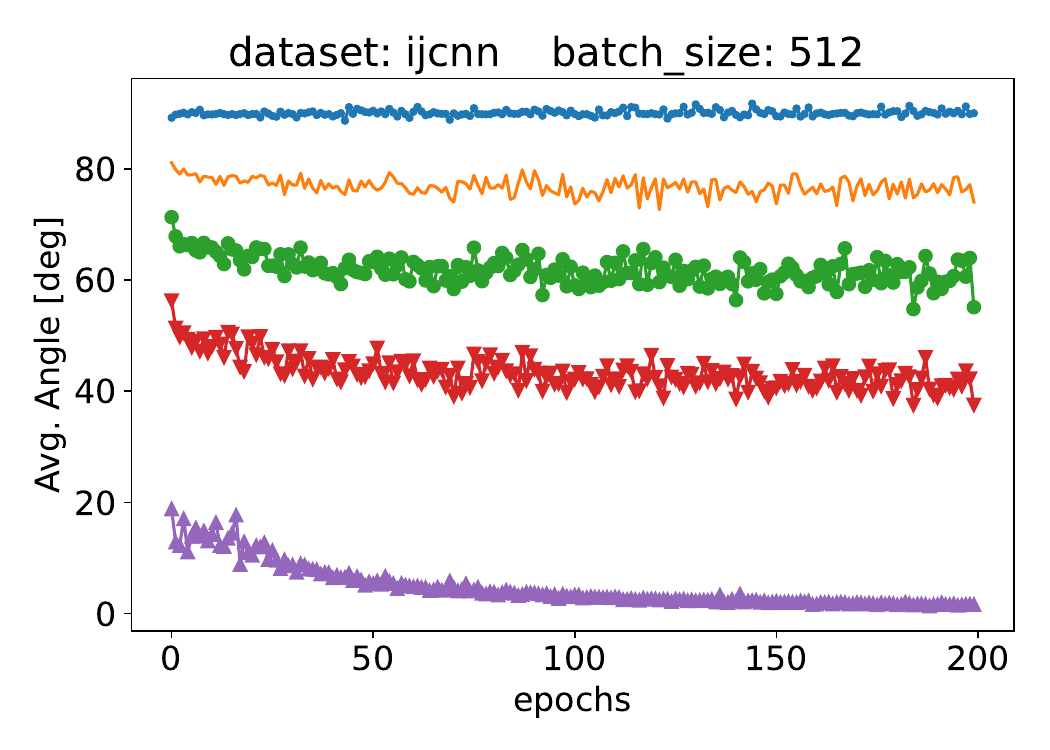}}
	\\
	\subfloat[Neural network classifier on \texttt{MNIST} - batch size 128.]{\includegraphics[width=.45\textwidth]{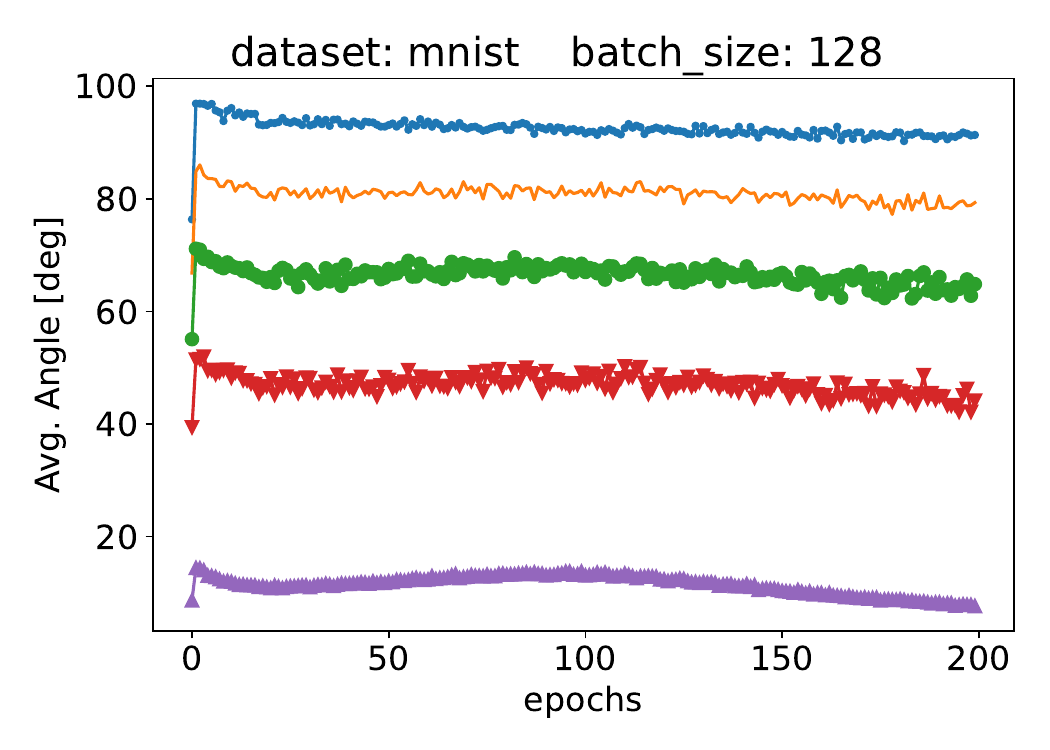}}
	\hfil
	\subfloat[Neural network classifier on \texttt{MNIST} - batch size 512.]{\includegraphics[width=.45\textwidth]{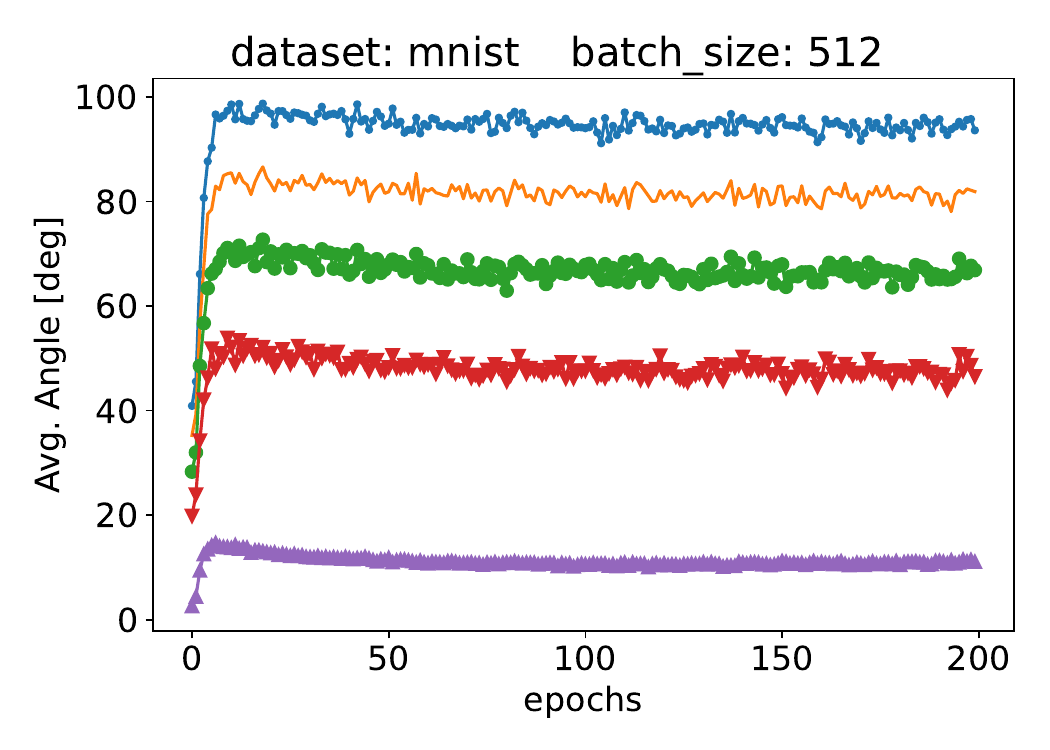}}	
	\caption{Average angle per mini-batch GD epoch (in degrees) between the momentum term and the negative stochastic gradient we would obtain for a new mini-batch with 0\%, 25\%, 50\%, 75\% and 100\% persistency.}
	\label{fig:overlap_analysis}
\end{figure}

\begin{table}[htbp]
	\caption{Total number of times, during training, that the momentum term is not a descent direction for a possible new mini-batch with various overlap percentages.}
	\label{tab:overlap_analysis}
	\centering
	\begin{tabular}{lcrrrrr}
		\toprule
		& & \multicolumn{5}{c}{Overlap}  \\
		\cmidrule(r){3-7}
		 Problem     & Batch Size & 0\% & 25\% & 50\% & 75\% & 100\%\\
		\midrule
		 \multirow{ 2}{*}{\texttt{ijcnn}}  & 128 & 5692 & 32 & 19 & 12 & 10   \\
		  & 512 & 1552 & 111 & 51 & 29 & 11  \\
		\midrule
		 \multirow{2}{*}{\texttt{MNIST}} & 128 & 54319 & 28282 & 10345 & 2069 & 5   \\
		  & 512 & 14745 & 6533 & 1325 & 64 & 0  \\
		\bottomrule
	\end{tabular}
\end{table}

We conclude the preliminary discussion on mini-batch persistency remarking that, in practice, any mini-batch GD type optimizer can benefit from overlapped mini-batches. 

\smallskip
A convenient solution to force a 50\% overlap in random reshuffling implementations \cite{gurbuzbalaban2021random} of mini-batch GD is to partition training data into $2N_b$ disjoint sets $\mathcal{R}^k_1, \dots, \mathcal{R}_{2N_b}^k$ at each epoch $\kappa$ and work, at iteration $j$ within the epoch, with the subset ${B}_{j}^\kappa = \mathcal{R}^\kappa_{j-1} \cup \mathcal{R}^\kappa_{j}$, so that, for all $j = 2, \ldots, 2N_b$, we have ${B}_{j}^\kappa \cap \mathcal{B}_{j+1}^\kappa = \mathcal{R}_{j}$ and ${B}_{j}^\kappa \cap{B}_{h}^\kappa =\emptyset$ for all $h>j+1$. With this approach, each data point is used twice (in subsequent iterations) within an epoch and the number of updates per epoch is doubled. Yet, given the sequential usage of the partitions $\mathcal{R}_j$, we can keep data in RAM/VRAM  without the need of reading the same data from disk twice per epoch. In other words, when overlap is used, the I/O time from disk required per epoch (which is one of the dominating components in the computational cost of algorithms in large-scale learning) does not increase w.r.t.\ the non-overlap case.

At this point, we find it useful to show how a 50\% mini-batch persistency affects SGD with momentum, Adam \cite{kingma2015}, PoNoS \cite{galli2023} and MSL\_SGDM \cite{fan2023}. To this aim, we look at the progress of the training loss through time on the \texttt{MNIST} problem already employed above. Each algorithm is run - for 3 different random seeds - for a  maximum of 200 epochs and is stopped earlier in case a solution $x^k$ s.t. $f(x^k) < 10^{-5}$ is found. Batch sizes of 128 and 512 are tested. The results shown in Figure \ref{fig:overlap_comparison} suggest that a 50\% overlap can lead to a significant improvement in the performance of these methods in the tasks considered. 

\begin{figure}[htbp]
	\centering
	\subfloat[SGD+M - bs: 128]{\includegraphics[width=.24\textwidth]{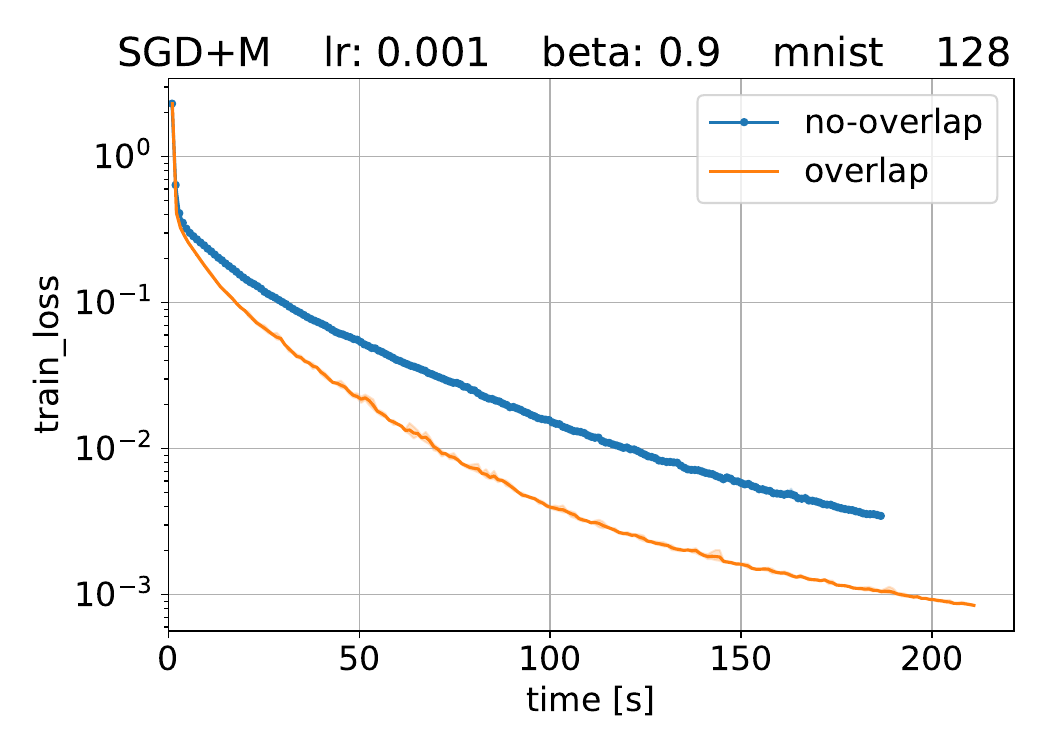}}
	\hfil
	\subfloat[Adam - bs: 128]{\includegraphics[width=.24\textwidth]{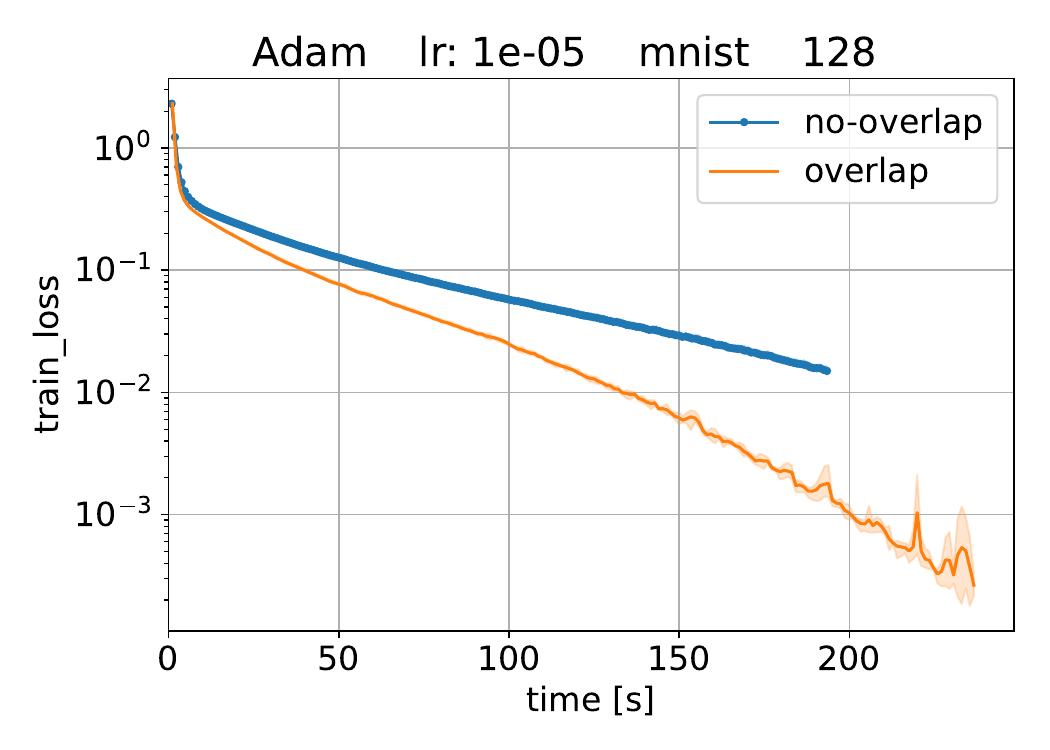}}
	\hfil\subfloat[PoNoS - bs: 128]{\includegraphics[width=.24\textwidth]{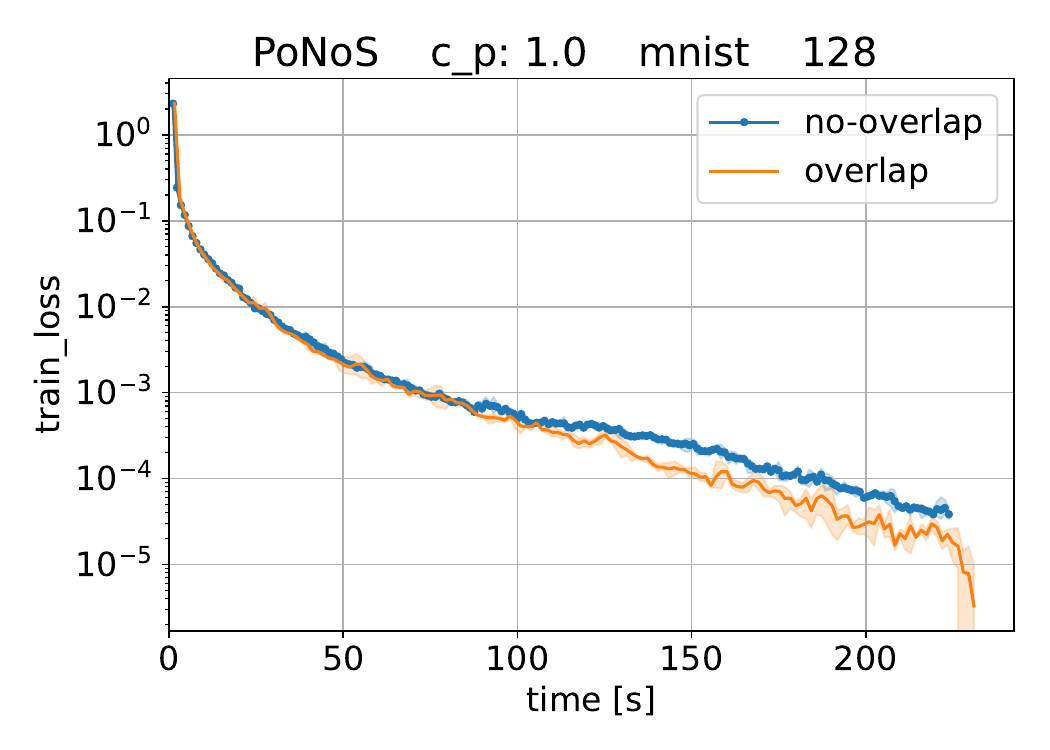}}
	\hfil\subfloat[MSL - bs: 128]{\includegraphics[width=.24\textwidth]{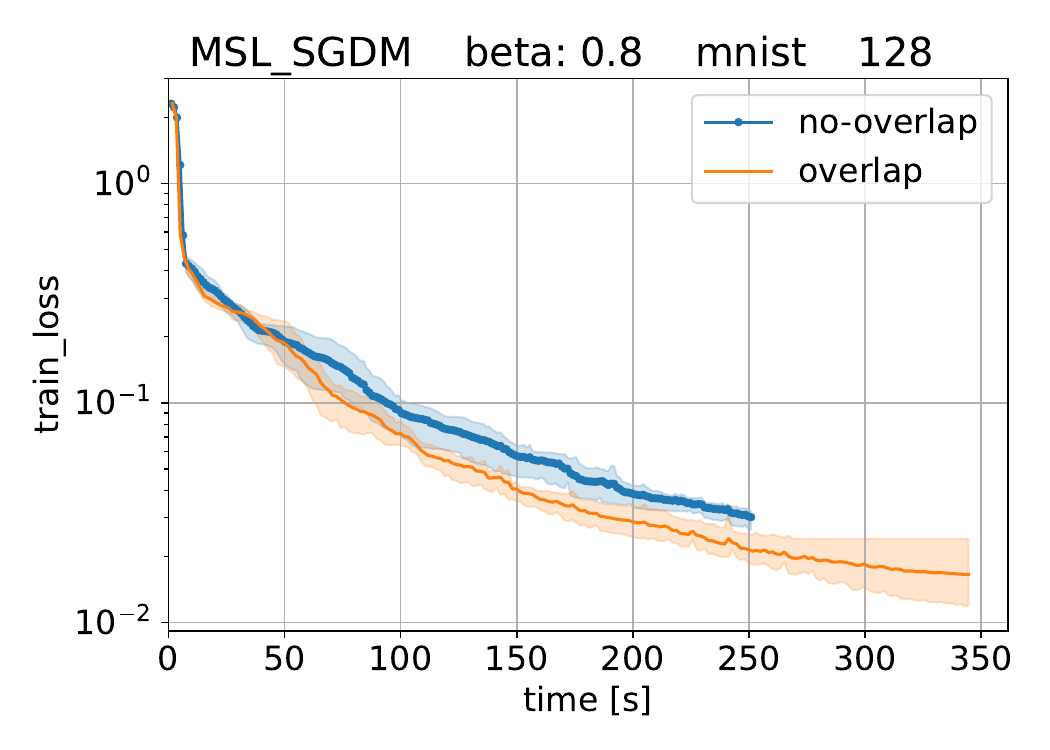}}
	\\
	\subfloat[SGD+M - bs: 512]{\includegraphics[width=.24\textwidth]{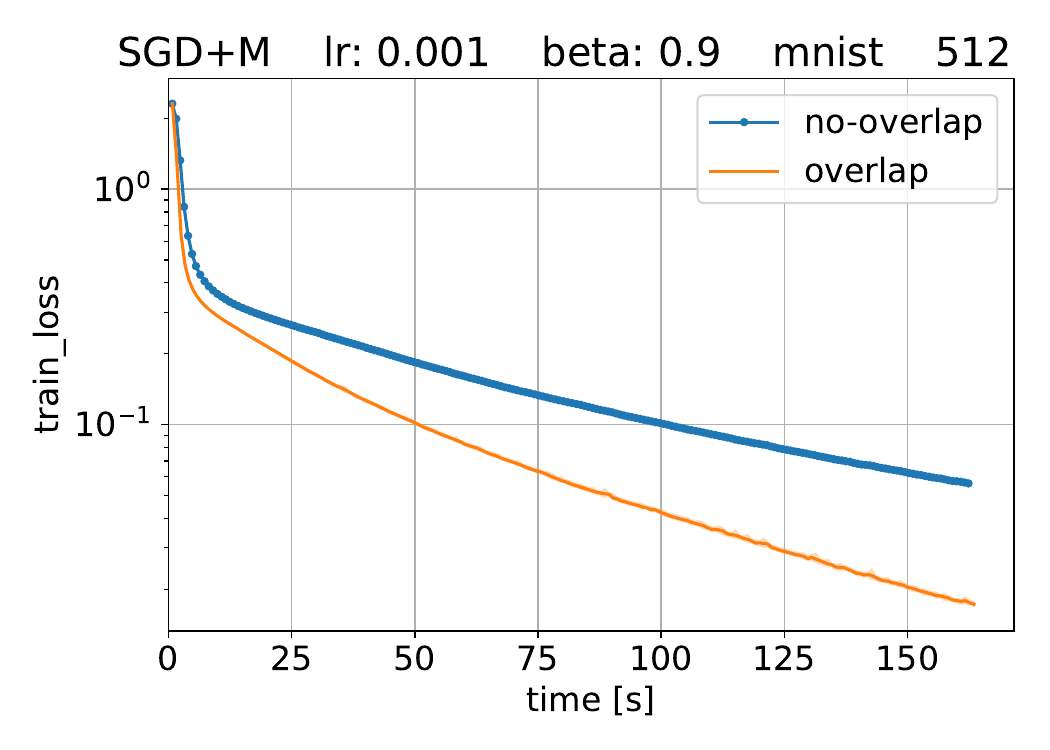}}
	\hfil
	\subfloat[Adam - bs: 512]{\includegraphics[width=.24\textwidth]{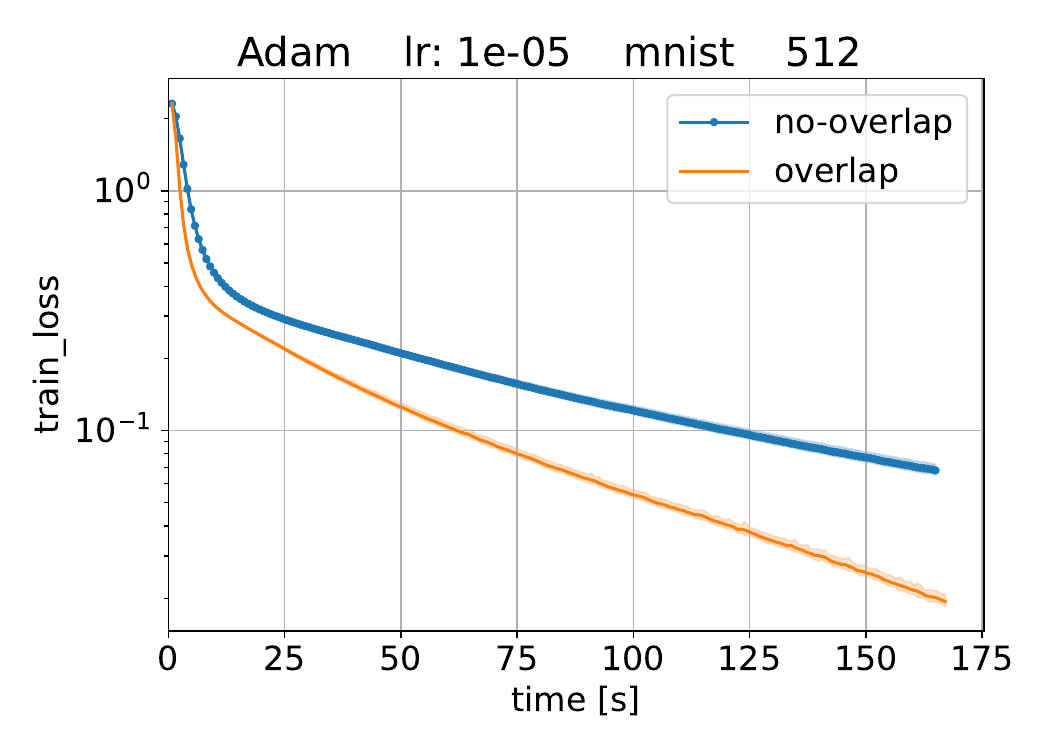}}
	\hfil\subfloat[PoNoS - bs: 512]{\includegraphics[width=.24\textwidth]{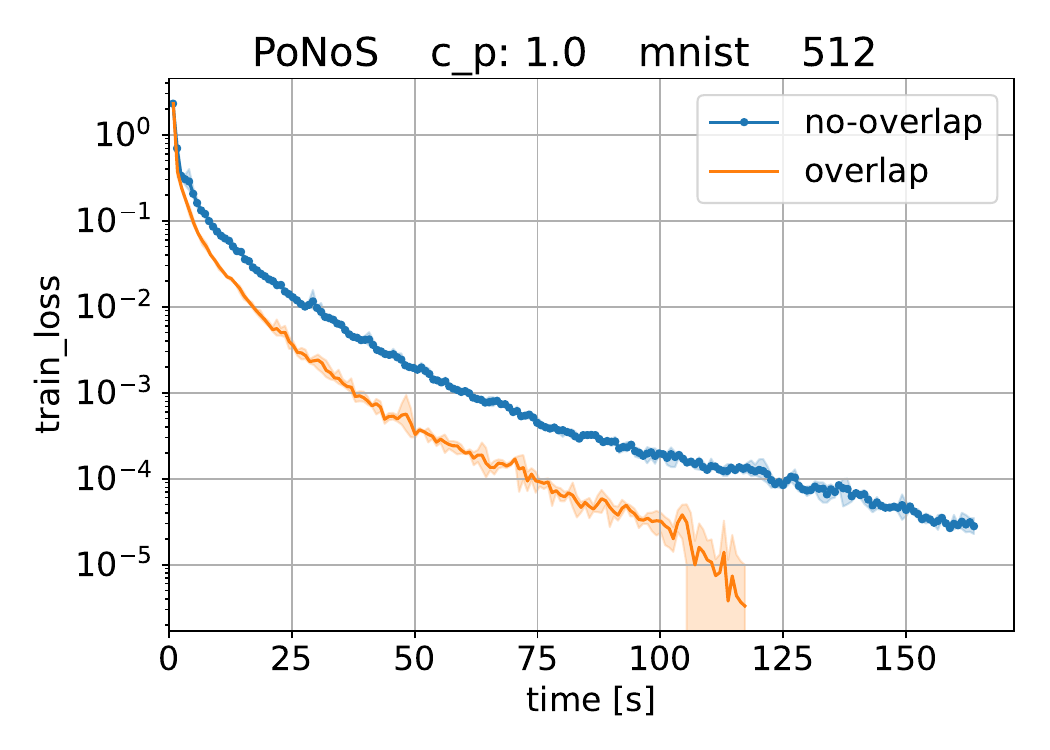}}
	\hfil\subfloat[MSL - bs: 512]{\includegraphics[width=.24\textwidth]{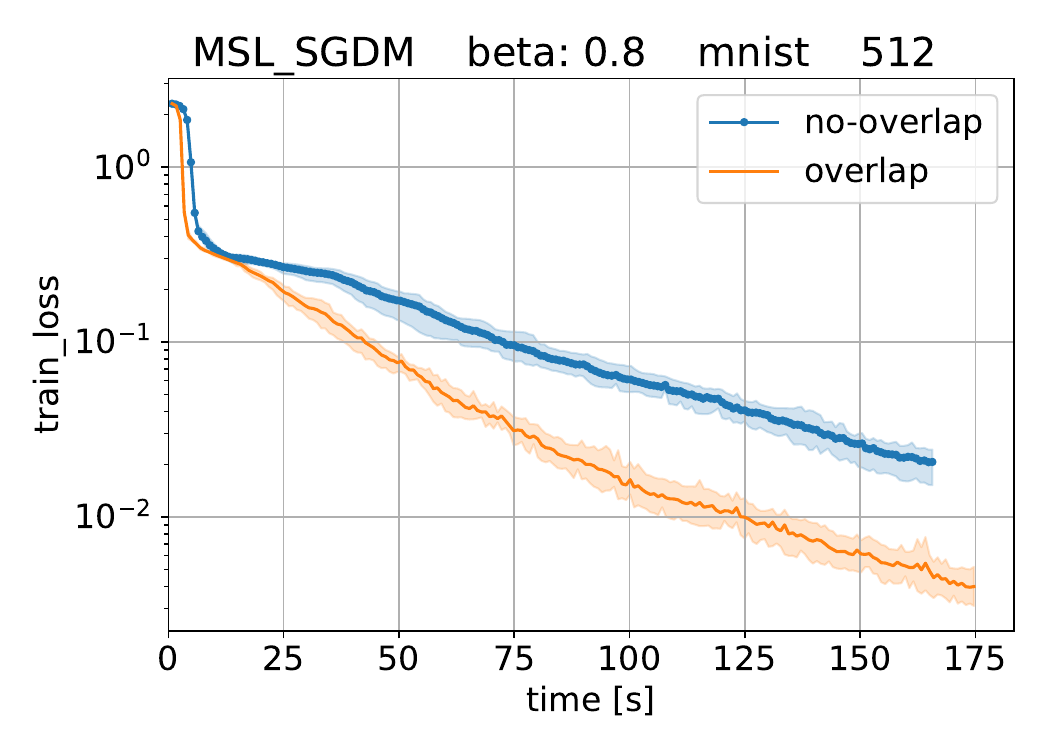}}
	\caption{Effect of a 50\% mini-batch persistency with state-of-the-art algorithms (Minibatch-GD with momentum, Adam, PoNoS, MSL-SGDM) applied to the problem of training a multi-layer perceptron on the \texttt{MNIST} dataset. Time is specified in seconds.}
	\label{fig:overlap_comparison}
\end{figure}

\section{Data-persistent Conjugate Gradient Rules}
\label{sec:cg}

The possibility of exploiting mini-batch persistency opens up a range of options to determine a suitable value for $\beta_k$ in heavy-ball, so as to 
\begin{enumerate}[(i)]
	\item obtain a descent direction for $f_{k}$ at $x^k$;
	\item get a large enough momentum term, not to shoot down its nice contribution for the overall optimization process.
\end{enumerate}

The particular strategy we propose in this paper stems from the strong connection existing between heavy-ball and nonlinear conjugate gradient (CG) methods. Indeed, in both classes of algorithms, the new search direction is a linear combination of the current gradients and the previous direction. 

\smallskip
The conceptual distinction between heavy-ball and CG is subtle and not unambiguously recognized by the optimization community. There is a substantial equivalence in the convex quadratic case, if heavy-ball parameters are chosen optimally \cite{polyak1987introduction}. In the general nonlinear case, CG methods define the search direction by means of a predetermined formula for $\beta_k$ and then carry out a Wolfe line search; on the other hand, heavy-ball selects $\alpha_k$ and $\beta_k$ altogether by some rule. If an Armijo-type line search is employed after choosing $\alpha_k$ and $\beta_k$ according to some rule, as done, e.g., in \cite{fan2023,lapucci2024globally},  we are arguably somewhere in a middle ground. We propose a framework that belongs to this middle ground in the stochastic setting.
\medskip
\begin{remark}
In the recent literature, two conjugate-gradient methods for finite-sum problem with convergence guarantees have been proposed \cite{jin2019stochastic, kou2023mini}. The two algorithms are substantially direct descendants of classical CG methods in nonlinear optimization, both possessing convergence properties in expectation. However, they present two major differences w.r.t.\ what we present here, which actually make their employment very unappealing in deep learning scenarios: they require a Wolfe-type line search, significantly increasing the per-iteration cost, and they employ variance reduction techniques, as the algorithm presented in \cite{jin2019stochastic} requires periodic full batch evaluations of the gradient, and the one presented in \cite{kou2023mini} stores per sample gradients, which have a massive impact on the memory footprint of the method.
\end{remark}
\medskip

In the general case, conjugate gradient methods perform a momentum update along a direction given by
$$d_k = -\nabla f(x^k)+\beta_k d_{k-1} = -\nabla f(x^k)+\frac{\beta_k}{\alpha_{k-1}}(x^k-x^{k-1}).$$
Several formulae have been proposed that can be used to compute $\beta_k$, that collapse to the same rule when $f$ is quadratic \cite{polyak1987introduction}. Given $y_k = \nabla f(x^k) - \nabla f(x^{k-1})$, some well-known formulations for computing $\beta_k$ (cfr.\ \cite[Sec.\ 12.5]{grippo2023introduction}) are:
\begin{equation}
	\label{eq:cg_rules}
	\begin{aligned}
		&\beta_k^{HS} = \frac{\nabla f(x^k)^T y_k}{(x^k - x^{k-1}) ^T y_k}, && \text{(Hestenes-Stiefel)} \\
		&\beta_k^{FR} = \frac{||\nabla f(x^k)||^2}{||\nabla f(x^{k-1})||^2}, && \text{(Fletcher-Reeves)} \\
		&\beta_k^{PPR} = \frac{\nabla f(x^k)^T y_k}{||\nabla f(x^{k-1})||^2}. && \text{(Polyak-Polak-Ribiére)}
	\end{aligned}
\end{equation}

We shall note that, with the above rules, we exploit the information available during iteration $k-1$ to define the coefficient $\beta_k$ for iteration $k$. Thus, in a stochastic scenario, rules of this kind would only be reasonable if $f_k$ did not change too much from an iteration to the other. This condition somewhat becomes true under mini-batch persistency.

Indeed, let ${B}_{k} \cap {B}_{k+1} = \mathcal{R}_{k} \neq \emptyset$ and $$f_{\mathcal{R}_k} = \frac{1}{|\mathcal{R}_k|}\sum_{i \in {\mathcal{R}_k}}f_i(x).$$ We can use the values of
$\nabla f_{\mathcal{R}_k} (x^{k})$ and $\nabla f_{\mathcal{R}_k} (x^{k+1})$, that are available during iteration $k$, to compute the value $\beta_{k+1}$ according to one of the above formulae. For instance, Fletcher-Reeves update becomes
$$\beta_{k+1}=\frac{\|\nabla f_{\mathcal{R}_k}(x^{k+1})\|^2}{\|\nabla f_{\mathcal{R}_k}(x^{k})\|^2}.$$

The momentum term at iteration $k+1$ will thus be suitably selected for the function $f_{\mathcal{R}_k}$, which actually constitutes a substantial portion of the function $f_{k+1}$ to be optimized at iteration $k+1$.

\smallskip

The advantage of this strategy is twofold: not only the estimated $\beta_k$ will be meaningful when the mini-batch changes, but we also do not need to make additional evaluations, since the four quantities involved in the CG-type formulae would be computed anyway throughout the algorithm. Note that using $\mathcal{R}_k$ to compute $\beta_{k+1}$ substantially differs from the strategy used in \cite{jin2019stochastic, kou2023mini}, as this allows to compute the CG update rule exactly for the function $f_{\mathcal{R}_k}$.

\section{The Algorithmic Framework}
\label{sec:alg}

It is clear that the strategy presented in Section \ref{sec:cg}, although reasonable, has a heuristic nature, since $\beta_k$ is estimated without looking at a portion (say a half) of currently considered data.
To obtain a sound algorithm, we must ensure that the obtained direction is actually of descent for $f_{k}$, before starting a line search procedure along the direction: if that were not the case, the backtracking algorithm could get stuck in an infinite loop.

Clearly, if $x^k-x^{k-1}$ is a direction of descent for $f_{k}$, i.e., $(x^k-x^{k-1})^T\nabla f_{k}(x^k)<0$, the descent property will hold for $d_k$ for any choice of $\beta_k\ge 0$.
From what we observed in the numerical results in Section \ref{sec:mb_pers}, this will often be the case in practice.
Yet, in the unfortunate case the direction $d$ obtained by the CG-type formula does not satisfy the above conditions, we shall resort to any of the following recovery strategies: 
\begin{itemize}
	\item switch to the negative stochastic gradient $-g_k(x^k) = -\nabla f_{k}(x^k)$ as a search direction;
	\item invert the search direction and take $d_k=-d$
	\item employ the damping strategy from \cite{fan2023};
	\item compute the direction according to the strategy proposed in \cite{lapucci2024globally}, i.e., 
	\begin{gather*}
		d_k = -ag_k(x^k)+b(x^k-x^{k-1}),\qquad (a,b)\in\argmin_{p\in\mathbb{R}^2} \frac{1}{2}p^TH_kp+p^T\begin{bmatrix}
			\|g_k(x^k)\|^2\\g_k(x^k)^T(x^k-x^{k-1})
		\end{bmatrix},
	\end{gather*}
	where the sequence of matrices $\{H_k\}\subseteq\mathbb{R}^{2\times 2}$ satisfies suitable uniform boundedness assumptions.
\end{itemize}
 
\medskip

\noindent The complete pseudocode of the resulting proposed method is reported in Algorithm \ref{alg::CG-noreshuffle}. 
A few aspects are worth to be remarked about the algorithm:
\begin{enumerate}
	\item In practical implementations, the algorithm will often be designed to work in a random-reshuffling fashion \cite{gurbuzbalaban2021random}. However, for the sake of simplicity and generality, we describe it in a ``pure'' stochastic setup. 
	\item At the first iteration, a pure SGD step is performed.
	\item The set $\mathcal{R}_k$ identifies the training samples that are reused in consecutive mini-batches, i.e. $B_k \cap B_{k+1} = \mathcal{R}_k$. In practice, at line \ref{step:R_k} of the algorithm, we can build $\mathcal{R}_k$ selecting from $B_k$ the indices that were drawn in the most recent iterations at step \ref{step:rand_draw}, so that during an epoch we avoid using the same training samples for many iterations in a row. For example, when using a $50\%$ overlap, we will have $\mathcal{R}_k = \mathcal{S}_k$ so as to use each training sample twice per epoch. As we highlight in Section \ref{sec:convergence}, using a deterministic rule at line \ref{step:R_k} is crucial to characterize the bias of the gradient estimate $g_k(x^k)$. 
	\item We do not specify the particular way of choosing the tentative step size $\alpha_0^k$ at line \ref{step:init_step}. There are various options that might work both theoretically and computationally. Of course, a fixed value could be used for all iterations; however, we might prefer to define it adaptively. For this latter approach, a heuristic method, such as the one proposed by Vaswani et al. \cite{vaswani2019}, could be considered, or we could use the generalized Stochastic Polyak step size introduced by Wang et al. \cite{wang2023}. 
	\item Finally, we shall underline that Algorithm \ref{alg::CG-noreshuffle} is shown to be using a monotone Armijo-type line search at step \ref{step:line_search}. This is done to ease the description of the method, but nothing prohibits to employ a nonmonotone line search.
\end{enumerate}

\begin{algorithm}[htbp]
	\caption{Mini-Batch Conjugate Gradient with Data Persistency  \texttt{(MBCG-DP)}}
	\label{alg::CG-noreshuffle}
	\begin{algorithmic}[1]
		\State{Input: $x^0 \in \mathbb{R}^n$, $f_1, \dots, f_N: \mathbb{R}^n \rightarrow \mathbb{R}$, $\gamma\in(0,1)$, $\delta\in(0,1)$.}
		\State Let $\mathcal{R}_{0}=\emptyset$, 
		\State $\beta_1=0$, $d_0 = 0$
		\For{$k=1,2,\ldots$}
		\State Randomly draw a subset $\mathcal{S}_k\subset\{1,\ldots,N\}\setminus\mathcal{R}_{k-1}$ \label{step:rand_draw}
		\State Set $B_k=\mathcal{R}_{k-1} \cup \mathcal{S}_{k}$
		\State Let $f_k = \frac{1}{|B_k|}\sum_{i \in {B_k}}
		f_{i}\;$ and $\;{g}_k = \frac{1}{|B_k|}\sum_{i \in {B_k}}\nabla
		f_{i}$
		\State Set $d = -{g}_k(x^k)+\beta_kd_{k-1}$
		\If{$ d^Tg_k(x^k)<0$ \label{step:if_sgr}}
		\State Set $d_k=d$
		\Else
		\State \label{step:sgr} Define $d_k$ such that $d_k^Tg_k(x^k)<0$
		\EndIf
		\State \label{step:init_step} Define a suitable initial step size $\alpha_0^k$
		\State Set $\alpha_k = \max_{j=0,1,\ldots}\{\alpha_0^k\delta^j\mid f_{k}(x_k+\alpha_0^k\delta^j d_k)\le f_{k}(x_k) +\gamma \alpha_0^k\delta^j g_k(x^k)^Td_k \}$ \label{step:line_search}
		\State Set $x_{k+1} = x_k+\alpha_k d_k$
		\State Select $\mathcal{R}_k \subset B_k$ according to a deterministic rule \label{step:R_k}
		\State  Define $\beta_{k+1}$ from $\nabla f_{\mathcal{R}_k}(x^k)$ and $\nabla f_{\mathcal{R}_k}(x^{k+1})$ according to any CG rule \eqref{eq:cg_rules}  \label{step:beta}
		\EndFor
	\end{algorithmic}
\end{algorithm}

\subsection{Convergence Analysis} \label{sec:convergence}
We now turn to the presentation of a formal convergence analysis for Algorithm \ref{alg::CG-noreshuffle}. The analysis tackles two main issues, that are addressed in two separate parts.

\subsubsection{Estimator Bias in Presence of Mini-Batch Persistency}
\label{sec:bias}
The key element of most convergence analyses for stochastic methods, including the one we carry out in the next subsection, relies on the availability of a conditionally unbiased estimator $g_k(x)$ of $\nabla f(x^k)$, i.e. $\mathbb{E}_k [g_k(x^k)] = \nabla f(x^k)$. We use $\mathbb{E}_k [\cdot]$ to denote the conditional expectation with respect to $x^k$, meaning $\mathbb{E}_k [\cdot] = \mathbb{E}[\cdot \mid x^k]$, while $\mathbb{E}[\cdot]$ indicates the total expectation.

However, in contrast to classical setups, where the unbiased estimator is reasonably obtained by drawing at each iteration the mini-batch to be used completely at random, we employ a data persistency strategy that turns this assumption false, as the set of indexes $\mathcal{R}_{k-1}$ is fixed, and thus deterministic, for the iteration $k$.

Let us denote $\bar{\mathcal{S}} = \{1,\ldots,N\}\setminus \mathcal{S}$. Recalling that $\nabla f_\mathcal{S}(x) = \frac{1}{|\mathcal{S}|}\sum_{i \in {\mathcal{S}}}\nabla f_i(x)$, and assuming for simplicity that $\mathcal{S}_k$ is randomly drawn from $\bar{\mathcal{R}}_{k-1}$ - so that $\mathbb{E}_k[\nabla f_{\mathcal{S}_k}(x^k)] = \nabla f_{\bar{\mathcal{R}}_{k-1}}(x^k)$ - we get that 
\begin{equation}
	\label{eq:biased_grad_calc}
	\begin{aligned}
	\mathbb{E}_k [g_k(x^k)] &= \mathbb{E}_k\left[\frac{1}{|B_k|}\sum_{i \in {B_k}}\nabla  f_i(x^k)\right]\\ 
	&= \mathbb{E}_k\left[\frac{1}{|B_k|}\left(\sum_{i \in {\mathcal{R}}_{k-1}}\nabla  f_i(x^k)+\sum_{i \in {\mathcal{S}}_{k}}\nabla  f_i(x^k)\right)\right]\\
	&= \frac{1}{|B_k|}\sum_{i \in {\mathcal{R}}_{k-1}}\nabla  f_i(x^k)+ \mathbb{E}_k\left[\frac{1}{|B_k|}\sum_{i \in {\mathcal{S}}_{k}}\nabla  f_i(x^k)\right]\\
	&= \frac{|\mathcal{R}_{k-1}|}{|B_k|}\nabla f_{\mathcal{R}_{k-1}}(x^k) + \frac{|\mathcal{S}_k|}{|B_k|}\mathbb{E}_k\left[\nabla  f_{\mathcal{S}_k}(x^k)\right]\\
	&= \frac{|\mathcal{R}_{k-1}|}{|B_k|}\nabla f_{\mathcal{R}_{k-1}}(x^k) + \frac{|\mathcal{S}_k|}{|B_k|}\nabla  f_{\bar{\mathcal{R}}_{k-1}}(x^k).
	\end{aligned}
\end{equation}
If, for instance, we assume $|\mathcal{R}_{k-1}| = |\mathcal{S}_{k}| = |B_k|/2$, we get 
$$\mathbb{E}_k [g_k(x^k)]=\frac{1}{2}\nabla f_{\mathcal{R}_{k-1}}(x^k) + \frac{1}{2}\nabla  f_{\bar{\mathcal{R}}_{k-1}}(x^k),$$
which is a gradient estimate that is clearly biased towards the ``persistent'' samples. When taking the expected value for the gradient estimator, the weight assigned to the terms $\nabla f_i(x^k)$ for $i\in \mathcal{R}_{k-1}$ is $\frac{1}{2|\mathcal{R}_{k-1}|}$, evidently higher than it would be in the common unbiased stochastic gradient, i.e., $\frac{1}{N}$.

Fortunately, we can recover unbiasedness of the estimator applying a suitable correction when constructing the approximation $g_k(x^k)$ - and similarly $f_k(x^k)$. This result is formalized hereafter.

\begin{proposition}
	\label{prop:unbiased}
	Let $\mathcal{R}_{k-1}\subset\{1,\ldots,N\}$ and let $\mathcal{S}_k$ be a random subset of $\bar{\mathcal{R}}_{k-1} = \{1,\ldots,N\}\setminus\mathcal{R}_{k-1}$. Let $\zeta_k=\frac{N-|\mathcal{R}_{k-1}|}{|\mathcal{S}_{k}|}$. Then, the quantities 
	$$
	f_k(x^k) = \frac{1}{N} \left(\sum_{i \in {\mathcal{R}}_{k-1}} f_i (x^k) + \zeta_k \sum_{i \in {\mathcal{S}_k}} f_i (x^k) \right)$$
	and
	$$g_k(x^k) 
		= \frac{1}{N} \left(\sum_{i \in {\mathcal{R}}_{k-1}} \nabla f_i (x^k) + \zeta_k \sum_{i \in {\mathcal{S}_k}} \nabla f_i (x^k) \right)$$ 
	are conditionally unbiased estimators of $f(x^k)$ and $\nabla f(x^k)$ given $x^k$, i.e., $\mathbb{E}_k[f_k(x^k)] = f(x^k)$ and $\mathbb{E}_k[g_k(x^k)] = \nabla f(x^k)$. 
\end{proposition}
\begin{proof}
	We only prove the result for $g_k$, as the proof for $f_k$ is substantially identical. Taking the conditional expected value of $g_k(x^k)$ given $x^k$, recalling that $ \frac{1}{|\mathcal{S}|}\sum_{i \in {\mathcal{S}}}\nabla f_i(x) = \nabla f_\mathcal{S}(x)$, that $\mathbb{E}_k[\nabla f_{\mathcal{S}_k}(x^k)] = \nabla f_{\bar{\mathcal{R}}_{k-1}}(x^k)$ and \eqref{eq:biased_grad_calc}, we get
	\begin{align*}
		\mathbb{E}_k [g_k(x^k)] &= \mathbb{E}_k\left[\frac{1}{N}\left(\sum_{i \in {\mathcal{R}}_{k-1}}\nabla  f_i(x^k)+\zeta_k\sum_{i \in {\mathcal{S}}_{k}}\nabla  f_i(x^k)\right)\right]\\
		&= \frac{|\mathcal{R}_{k-1}|}{N}\nabla f_{\mathcal{R}_{k-1}}(x^k) + \zeta_k\frac{|\mathcal{S}_k|}{N}\nabla f_{\bar{\mathcal{R}}_{k-1}}(x^k)\\
		&=\frac{1}{N}\left(\sum_{i \in {\mathcal{R}}_{k-1}}\nabla  f_i(x^k)+\frac{N-|\mathcal{R}_{k-1}|}{|\mathcal{S}_{k}|}\frac{|\mathcal{S}_{k}|}{N-|\mathcal{R}_{k-1}|}\sum_{i \in {\bar{\mathcal{R}}}_{k-1}}\nabla  f_i(x^k)\right)\\
		&=\frac{1}{N}\sum_{i =1}^{N}\nabla  f_i(x^k)  = \nabla f(x^k),
	\end{align*}
	thus getting the thesis.
\end{proof}

\smallskip
\begin{remark}
	The correction to make the gradient estimate unbiased even in presence of mini-batch persistency allows to soundly employ the latter technique from a theoretical convergence standpoint. However, we note that the correction heavily alleviates the sought effects that drove the introduction of the overlap strategy itself in the first place, since the weight assigned to the ``persistent'' samples is drastically lowered. This is also evident in the computational experiments, as shown later in Figure \ref{fig:biased_mnist}. Therefore, we believe the theoretical analysis of convergence with a biased estimator would also be of great interest. Yet, this issue is a large open problem and its treatment would go far beyond the scope of this work; we thus leave it to future studies.
\end{remark}

\subsubsection{Properties of the Search Direction}
\label{sec:direction}
Together with the assumption on gradient estimate unbiasedness, largely discussed in the previous subsection, we need to state another important assumption for properly analyzing the algorithm: the search directions $d_k$ shall in fact be somehow tied to the true gradient $\nabla f(x^k)$. As pointed out in well-established literature (see, e.g., \cite[sec.\ 4.1]{bottou2018}), the sequence of search directions $\{d_k\}$ shall be such that, for all $k$,
$$\mathbb{E}_k[\|d_k\|]\le C_1\|\nabla f(x^k)\|,\quad \text{ and }\quad\mathbb{E}_k[d_k]^T\nabla f(x^k)\le - C_2 \color{black} \|\nabla f(x^k)\|^2.$$ 
Following the discussion in \cite{lapucci2024convergenceconditionsstochasticline}, this can be ensured if, given constants $c_1$, $c_2$ and $c_3$, the entire sequence $\{d_k\}$ satisfies the following set of conditions:
\begin{gather}
	\|d_k\|\le c_1\|g_k(x^k)\|,\label{eq:sgr1}\\
	d_k^T g_k(x^k)\le -c_2\|g_k(x^k)\|^2,\label{eq:sgr2}\\
	\operatorname{Var}_k(d^k)\le c_3\operatorname{Var}_k(g_k(x^k)),\label{eq:sgr3}
\end{gather}
where by $\operatorname{Var}_k$ we denote the conditional variance of a vector, i.e., $\operatorname{Var}_k(v) = \mathbb{E}_k[\|v\|^2]-\|\mathbb{E}_k[v]\|^2$.

Properties \eqref{eq:sgr1}-\eqref{eq:sgr2} can be checked, as safeguards, at each iteration for fixed values of $c_1$ and $c_2$: when not satisfied by the CG type direction, the algorithm can switch to the stochastic gradient direction. 

In order to ensure that condition \eqref{eq:sgr3} is also satisfied by the direction resulting from the safeguard condition, we need to keep parameter $\beta_k$ bounded; this is also not particularly restrictive, as it can be enforced by a clipping operation at the update time; in other words, we can simply set  $\beta_{k+1} = \min\{\hat{\beta}_{k+1},\bar{\beta}\}$, where $\hat{\beta}_{k+1}$ is the value resulting from the employed CG update rule. 
With this precaution, we only have to rely on an assumption on the variance of $\beta_{k+1}$, which appears reasonable as argued in \cite[Remark 2]{lapucci2024convergenceconditionsstochasticline}.

The satisfaction of condition \eqref{eq:sgr3} for the whole sequence, if clipping and safeguards are suitably applied, is formally stated and proved in the next proposition.

\begin{proposition} \label{prop:var_d_k}
	Let, for all $k$, $d_k = -g_k(x^k) + \beta_k d_{k-1}$, where 
	$$\beta_k = 
	\begin{cases}
		\min\{\hat{\beta}_{k},\bar{\beta}\} & \text{if } \|p_k\|\le c_1\|g_k(x^k)\| \text{ and } p_k^T g_k(x^k)\le -c_2\|g_k(x^k)\|^2, \\
		0 & \text{otherwise},
	\end{cases}$$
	with $p_k = -g_k(x^k) + \min\{\hat{\beta}_{k},\bar{\beta}\} d_{k-1}$. Assuming that, for all $k$, $\|d_{k-1}\| \leq M $ and $\operatorname{Var}_k (\beta_k) \leq \hat{c}_3 \operatorname{Var}_k (g_k(x^k))$ for some constants $M, \hat{c}_3 > 0$, we have that $\operatorname{Var}_k(d^k)\le c_3\operatorname{Var}_k(g_k(x^k))$, with $c_3=\left(1 + M^2 \hat{c}_3 + 2M \sqrt{\hat{c}_3}\right)$. 
\end{proposition}

\begin{proof}
	The variance of the direction $d_k$ is given by
	\begin{align*}
		\operatorname{Var}_k(-g_k(x^k) + \beta_k d_{k-1}) = \operatorname{Var}_k(g_k(x^k)) + \|d_{k-1}\|^2 \operatorname{Var}_k (\beta_k) - 2 \operatorname{Cov}_k(g_k(x^k), \beta_k d_{k-1}) 
	\end{align*}
	Exploiting the Cauchy-Schwartz inequality on the dot product $u^T v$ where $u_i = \sqrt{\operatorname{Var}_k((g_k(x^k))_i)}$ and $v_i=|(d_{k-1})_i|$, we can write 
	\begin{align*}
		|\operatorname{Cov}_k(g_k(x^k),\beta_k d_{k-1})| &\le \sum_{i=1}^{n} \left|\operatorname{Cov}_k((g_k(x^k))_i,\beta_k (d_{k-1})_i)\right|\\
		&\le \sum_{i=1}^{n} \sqrt{\operatorname{Var}_k((g_k(x^k))_i)}\sqrt{\operatorname{Var}_k(\beta_k (d_{k-1})_i)}\\
		& = \sqrt{\operatorname{Var}_k(\beta_k)}\sum_{i=1}^{n} \sqrt{\operatorname{Var}_k((g_k(x^k))_i)} |(d_{k-1})_i|\\
		&\le \sqrt{\operatorname{Var}_k(\beta_k)}\sqrt{\operatorname{Var}_k(g_k(x^k))}\|d_{k-1}\|.
	\end{align*}
	Using that $\|d_{k-1}\| \leq M$ and that $\operatorname{Var}_k (\beta_k) \leq \hat{c}_3 \operatorname{Var}_k (g_k(x^k))$ we obtain
	\begin{align*}
		\operatorname{Var}_k(d_k) &\le \operatorname{Var}_k(g_k(x^k)) + \|d_{k-1}\|^2 \operatorname{Var}_k (\beta_k) + 2 \sqrt{\operatorname{Var}_k(\beta_k)}\sqrt{\operatorname{Var}_k(g_k(x^k))}\|d_{k-1}\| \\
		& \leq \left(1 + M^2 \hat{c}_3 + 2M \sqrt{\hat{c}_3}\right) \operatorname{Var}_k(g_k(x^k)),
	\end{align*}
	which concludes the proof.
\end{proof}

Proposition \ref{prop:var_d_k} provides the final building block for establishing convergence results for the proposed algorithm. Note that the first assumption on $\|d_{k-1}\|$ is not restrictive, as it can be enforced at iteration $k$ by simple clipping of the then deterministic direction $d_{k-1}$ (i.e., using $\min\{1,\frac{M}{\|d_{k-1}\|}\}d_{k-1}$ instead of $d_{k-1}$). 

\subsubsection{Convergence Results}
\label{sec:main_theorem}
We are now ready to present the asymptotic convergence analysis for Algorithm \ref{alg::CG-noreshuffle}. In particular, we are showing how the method proposed in this paper possesses analogous properties as the other well-known stochastic line-search based methods from the literature \cite{vaswani2019fast, vaswani2019, loizou2021, galli2023, lapucci2024convergenceconditionsstochasticline}. Specifically, it is possible to guarantee fast convergence of the algorithmic framework under interpolation and PL assumptions \cite{ma2018, mishkin2020interpolation, polyak1987introduction}.

We recall that the minimizer interpolation property \cite[Def.\ 3]{mishkin2020interpolation} states that $$x^* \in \argmin_{x^* \in \mathbb{R}^n} f(x) \implies x^* \in \argmin_{x^* \in \mathbb{R}^n} f_i(x) \text{ for all } i \in \{1, \dots, N\}.$$
On the other hand, we say that a function $f$ satisfies the PL condition if $$\exists \mu > 0 \text{ such that } 2 \mu (f(x) - f^*) \leq ||\nabla f(x)||^2 \text{ for all } x \in \mathbb{R}^n.$$
We shall also recall that (see, e.g., \cite[Sec.\ 3]{lapucci2024convergenceconditionsstochasticline}), $L$-smoothness, minimizers interpolation and the PL condition imply the \textit{Strong Growth Condition} to hold for some $\rho>0$, i.e.,
$$\mathbb{E}_k [\|\nabla f_k (x)\|^2] \leq \rho \|\nabla f(x)\|^2 \text{ for all } x \in \mathbb{R}^n \text{ and for any } k.$$

Since we need to insert into the algorithmic framework the fixes described in sections \ref{sec:bias} and \ref{sec:direction}, we report for the sake of clarity the instructions of the provably convergent version of the algorithm in Algorithm \ref{alg::conv}. Note that we also made explicit the constant lower bound required for the initial step size $\alpha_k^0\ge \bar{\alpha}$ of the line search procedure.

\begin{algorithm}[htbp]
	\caption{Convergent \texttt{MBCG-DP}} 
	\label{alg::conv}
	\begin{algorithmic}[1]
		\State{Input: $x^0 \in \mathbb{R}^n$, $f_1, \dots, f_N: \mathbb{R}^n \rightarrow \mathbb{R}$, $\gamma\in(0,1)$, $\delta\in(0,1)$, $c_1,c_2,\bar{\alpha},\bar{\beta},M>0$.}
		\State Let $\mathcal{R}_{0}=\emptyset$, 
		\State $\beta_1=0$, $d_0 = 0$
		\For{$k=1,2,\ldots$}
		\State Randomly draw a subset $\mathcal{S}_k\subset\{1,\ldots,N\}\setminus\mathcal{R}_{k-1}$ 
		\State Set $B_k=\mathcal{R}_{k-1}\cup \mathcal{S}_{k}$
		\State Set $\zeta_k =\frac{N-|\mathcal{R}_{k-1}|}{|\mathcal{S}_{k}|} $
		\State Let $f_k =  \frac{1}{N} \left(\sum_{i \in {\mathcal{R}}_{k-1}} f_i + \zeta_k \sum_{i \in {\mathcal{S}_k}} f_i \right)$ 
		\State Let $g_k=\frac{1}{N} \left(\sum_{i \in {\mathcal{R}}_{k-1}} \nabla f_i + \zeta_k \sum_{i \in {\mathcal{S}_k}} \nabla f_i \right)$
		\State Set $d = -{g}_k(x^k)+\beta_k \min \left\{1, \frac{M}{\|d_{k-1}\|}\right\} d_{k-1} $ 
		\If{$\|d\|\le c_1\|g_k(x^k)\|$ \textbf{and} $d^Tg_k(x^k)\leq-c_2\|g_k(x^k)\|^2$ \label{step:safe}}
		\State Set $d_k=d$
		\Else
		\State  Set $d_k=-g_k(x^k)$ \label{step:end_safe}
		\EndIf
		\State  Define an initial step size $\alpha_0^k\ge \bar{\alpha}$
		\State Set $\alpha_k = \max_{j=0,1,\ldots}\{\alpha_0^k\delta^j\mid f_{k}(x_k+\alpha_0^k\delta^j d_k)\le f_{k}(x_k) +\gamma \alpha_0^k\delta^j g_k(x_k)^Td_k \}$ 
		\State Set $x_{k+1} = x_k+\alpha_k d_k$
		\State Select $\mathcal{R}_k \subset B_k$ according to a deterministic rule
		\State  Define $\hat{\beta}_{k+1}$ from $\nabla f_{\mathcal{R}_k}(x^k)$ and $\nabla f_{\mathcal{R}_k}(x^{k+1})$ according to any CG rule \eqref{eq:cg_rules} 
		\State Set $\beta_{k+1} = \min\{\hat{\beta}_{k+1},\bar{\beta}\}$
		\EndFor
	\end{algorithmic}
\end{algorithm}

The main convergence theorem is stated hereafter. The result is recovered by \cite[{Th.\ 5.3}]{lapucci2024convergenceconditionsstochasticline}.
\begin{theorem}
	Assume $f:\mathbb{R}^n\to\mathbb{R}$ is an $L$-smooth function, satisfying the PL condition with constant $\mu$ and the minimizer interpolation property, and that the randomly drawn functions $f_k$ considered at any iteration $k$ are $L_k$-smooth functions, with $L\le L_\text{max} = \max_{k}L_k$. 
	Let the sequence $\{\beta_k\}$ be such that $\operatorname{Var}_k (\beta_k) \leq \hat{c}_3 \operatorname{Var}_k (g_k(x^k))$ and that $c_2>c_3\left(1-\frac{1}{\rho}\right)$, being $c_3 = \left(1 + M^2 \hat{c}_3 + 2M \sqrt{\hat{c}_3}\right)$ and $\rho$ the SGC constant.
	Then, if $\{x^k\}$ is the sequence produced by Algorithm \ref{alg::conv} the following property holds:
	\begin{equation}
		\label{eq:main_thesis}
		\mathbb{E}[f(x^{k+1}) - f(x^*)] \leq (\eta\bar{\alpha})^k(f(x^0)-f(x^*)),
	\end{equation}
	where $\eta= \left(\frac{L_{\text{max}} c_1^2}{2 c_2}\left(\frac{1}{\gamma}+\frac{1}{\delta(1-\gamma)}\right) - 2 \bigg(c_2-c_3\bigg(1-\frac{1}{\rho}\bigg)\bigg) \mu   \right)$.
	
	\noindent Consequently, when all the constants involved are such that
	$0<\eta<\frac{1}{\bar{\alpha}}$, the convergence rate is linear with an $\mathcal{O}(\log(\frac{1}{\epsilon}))$ iteration, stochastic function evaluations and stochastic gradient evaluations complexity to obtain an $\epsilon$-accurate solution in expectation.
\end{theorem}
\begin{proof}
	The sequence of directions $\{d_k\}$ satisfies the conditions \eqref{eq:sgr1}-\eqref{eq:sgr2} with constants $c_1$ and $c_2$ as a result of steps \ref{step:safe}-\ref{step:end_safe}. By Proposition \ref{prop:var_d_k} and the instructions of the algorithm, it also satisfies condition \eqref{eq:sgr3} with $c_3$ defined as in the statement of the theorem. All the assumptions of \cite[Theorem 5.3]{lapucci2024convergenceconditionsstochasticline} are thus satisfied by the ones here, and we consequently get the thesis.	 
	
\end{proof}

\color{black}

\section{Computational Experiments}
\label{sec:exp}
In this section, we present the results of a comprehensive experimental analysis of the proposed Mini-Batch Conjugate Gradient method with Data Persistency (MBCG-DP). As previously discussed in Section \ref{sec:convergence}, the bias correction term required to make the gradient estimate unbiased actually opposes to the effect we introduce with the overlap strategy (for completeness, we show in Figure \ref{fig:biased_mnist} that bias correction indeed slows down the algorithm on the \texttt{MNIST} task considered in Section \ref{sec:mb_pers}). In the following computational experiments we therefore consider the algorithm version without any bias correction. 
All experiments reported in this section were implemented in \texttt{Python3}\footnote{The code for the proposed method is available at \href{https://github.com/dadoPuccio/MB-Conjugate-Gradient-DP}{github.com/dadoPuccio/MB-Conjugate-Gradient-DP}} and run on a machine with the following characteristics: Ubuntu 22.04, Intel i5-13400F Processor, 32 GB RAM, Nvidia GeForce RTX 4060 8GB.

\begin{figure}[htbp]
	\centering
	\includegraphics[width=.4\textwidth]{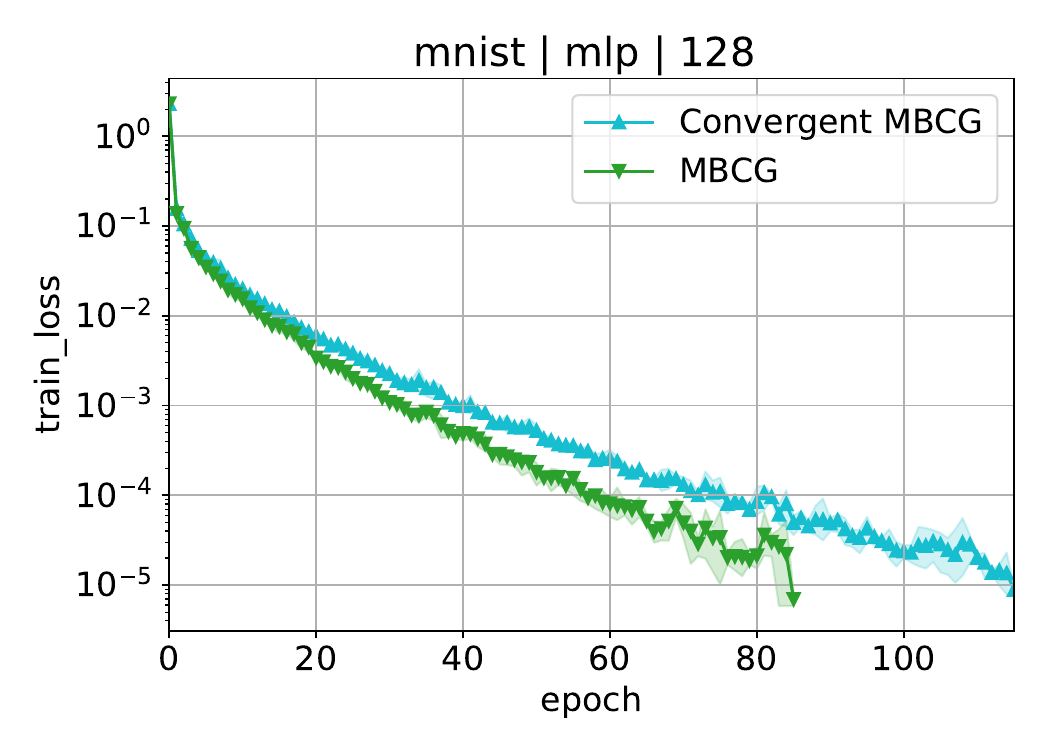}
	\hfil
	\includegraphics[width=.4\textwidth]{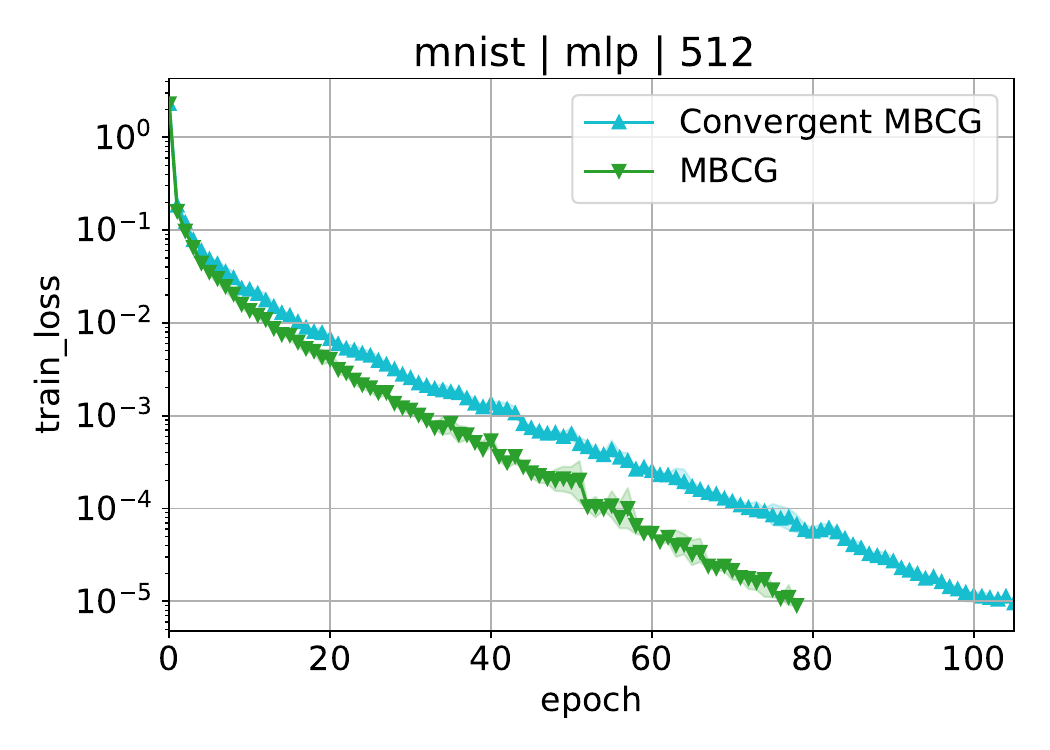}
	\caption{Behavior of MBCG-DP and its unbiased variant when employed to train a multi-layer perceptron on \texttt{MNIST} dataset with batch size of 128 and 512.}
	\label{fig:biased_mnist}
\end{figure}

The initial analysis focuses on the behavior of MBCG-DP when using different configurations. Specifically, we aim at identifying
\begin{enumerate}[(a)]
	\item the best CG rule to select $\beta_k$,
	\item the optimal strategy for selecting the tentative step size $\alpha_0^k$ in the line-search,
	\item the most effective method for modifying $d_k$ when it is not a descent direction for the current mini-batch objective.
\end{enumerate}

We consider in this phase the same two learning problems as in Section \ref{sec:mb_pers}: an RBF-kernel classifier on the \texttt{ijcnn} dataset and a multi-layer perceptron on the \texttt{MNIST} dataset.
The algorithm is run in a random-reshuffling fashion, with 50\% mini-batch persistency, and batch sizes of 128 and 512 are considered for both tasks. Standard parameters are employed for the nonmonotone Armijo-type line search, with $\gamma=0.5$, $\delta=0.5$. For each experiment 3 different seeds are considered.

In Figure \ref{fig:CG_rule_comparison} we reported the behavior of MBCG-DP when comparing the CG rules of Fletcher-Reeves (FR), Polyak-Polak-Ribiére (PPR) and Hestenes-Stiefel (HS). In this experiment the generalized SPS \eqref{eq:sps_momentum} with $c=1$ is used to select $\alpha_0^k$ and the momentum clipping approach (contraction factor 0.5) from \cite{fan2023} is employed in case $d_k$ is not a descent direction. 
The Fletcher-Reeves rule clearly resulted the most effective one in both tasks and for both batch sizes.

\begin{figure}[htbp]
	\centering
	
	\subfloat{\includegraphics[width=.32\textwidth]{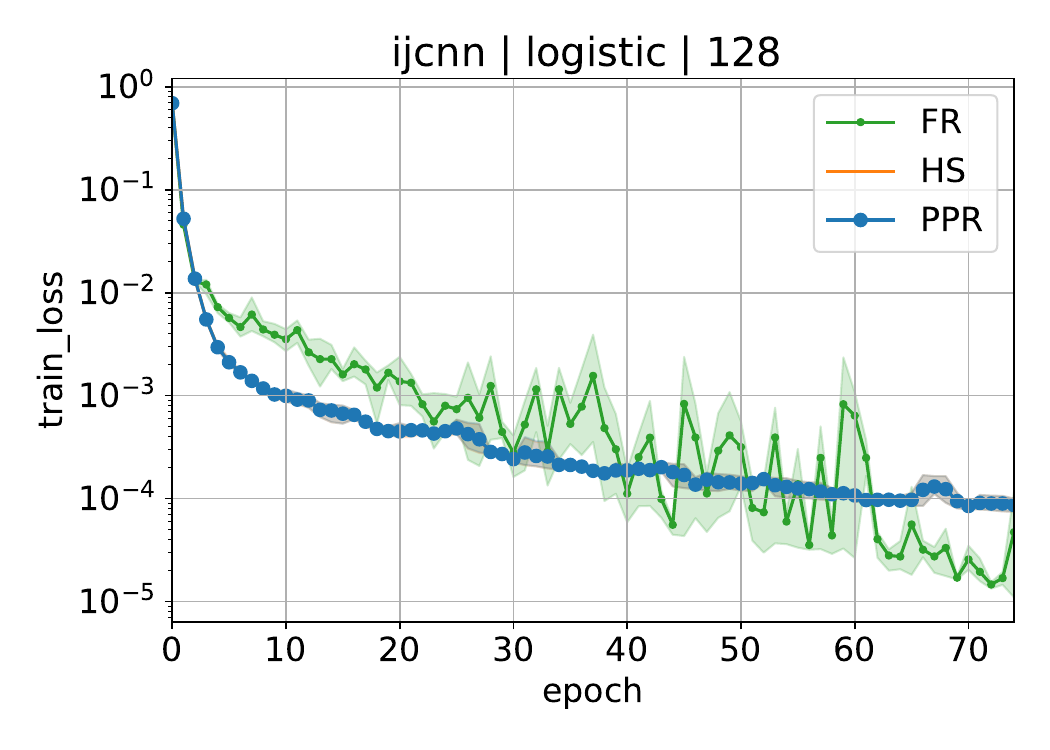}}
	\hfill
	\subfloat{\includegraphics[width=.32\textwidth]{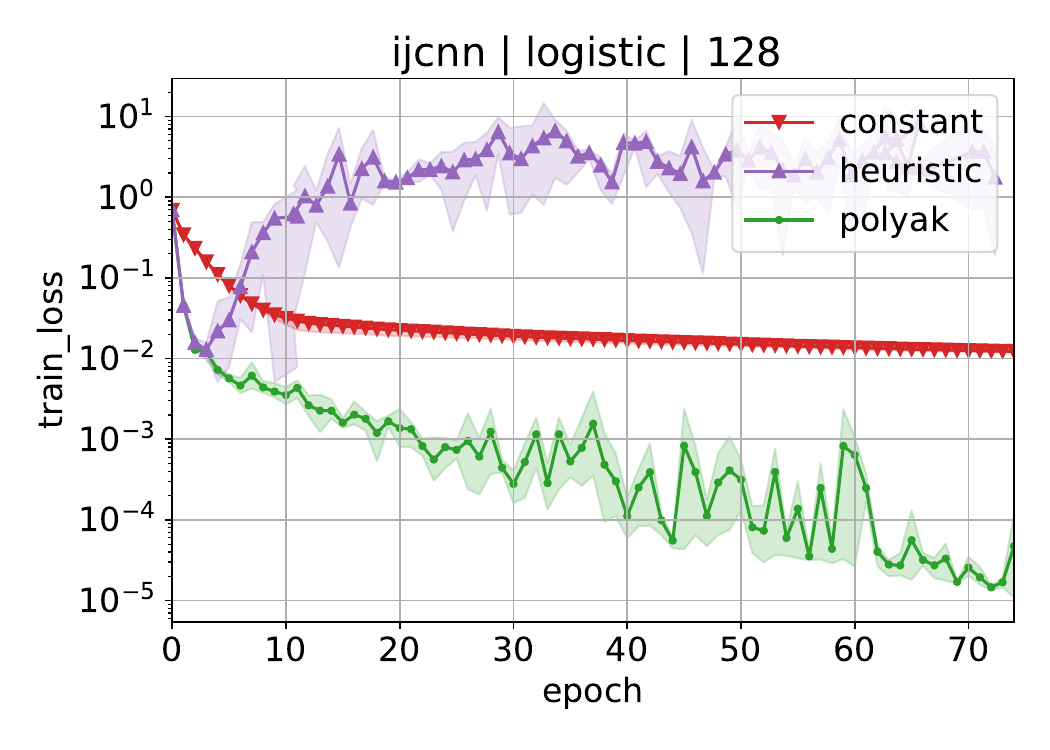}}
	\hfill
	\subfloat{\includegraphics[width=.32\textwidth]{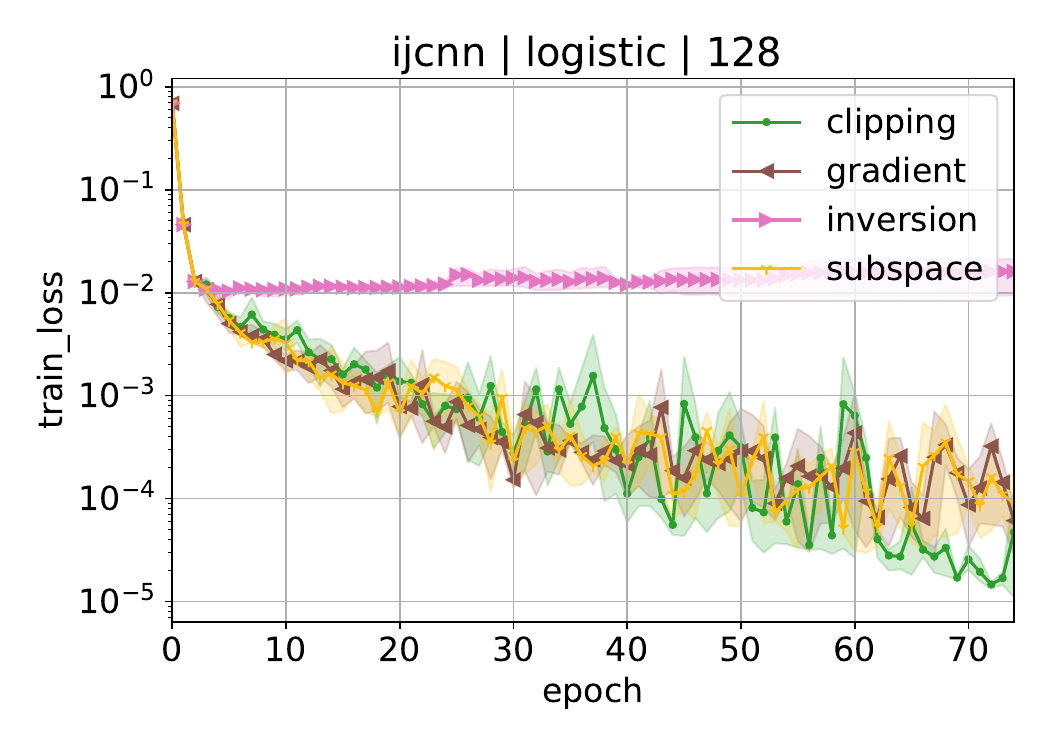}}
	\\
	\subfloat{\includegraphics[width=.32\textwidth]{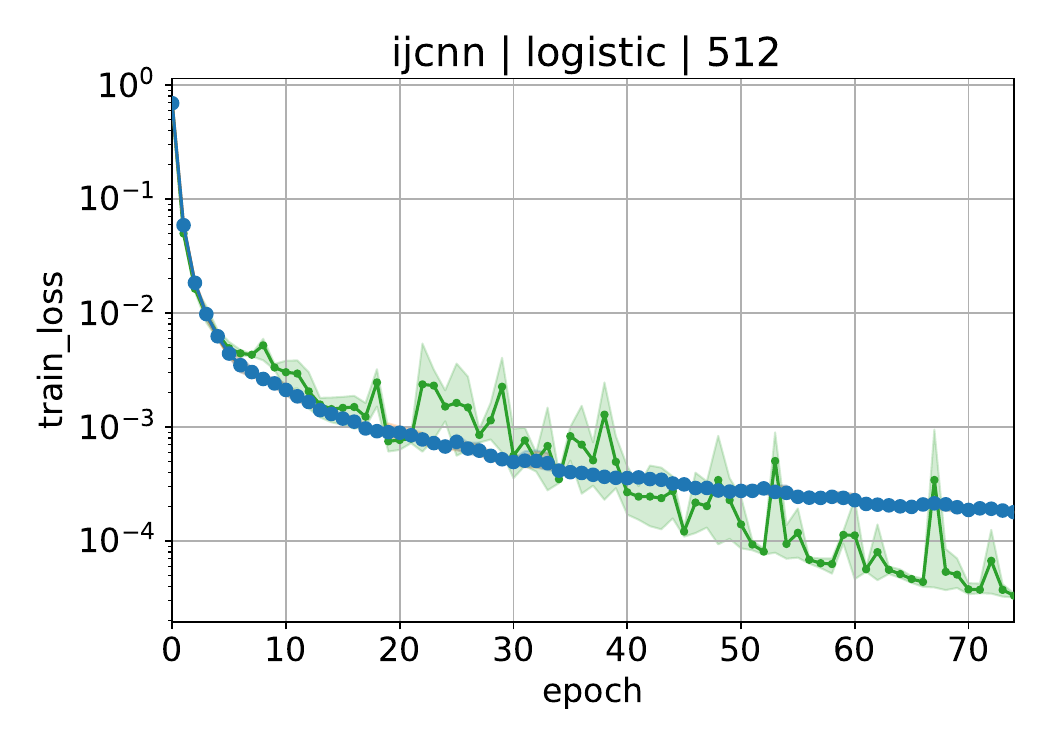}}
	\hfill
	\subfloat{\includegraphics[width=.32\textwidth]{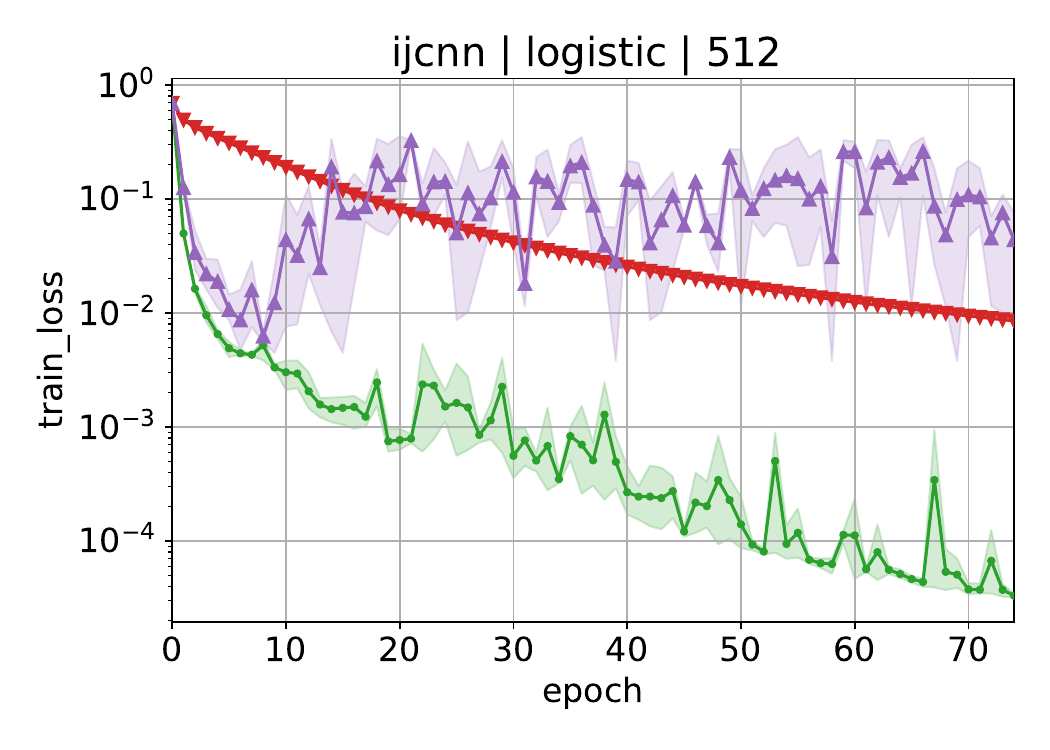}}
	\hfill
	\subfloat{\includegraphics[width=.32\textwidth]{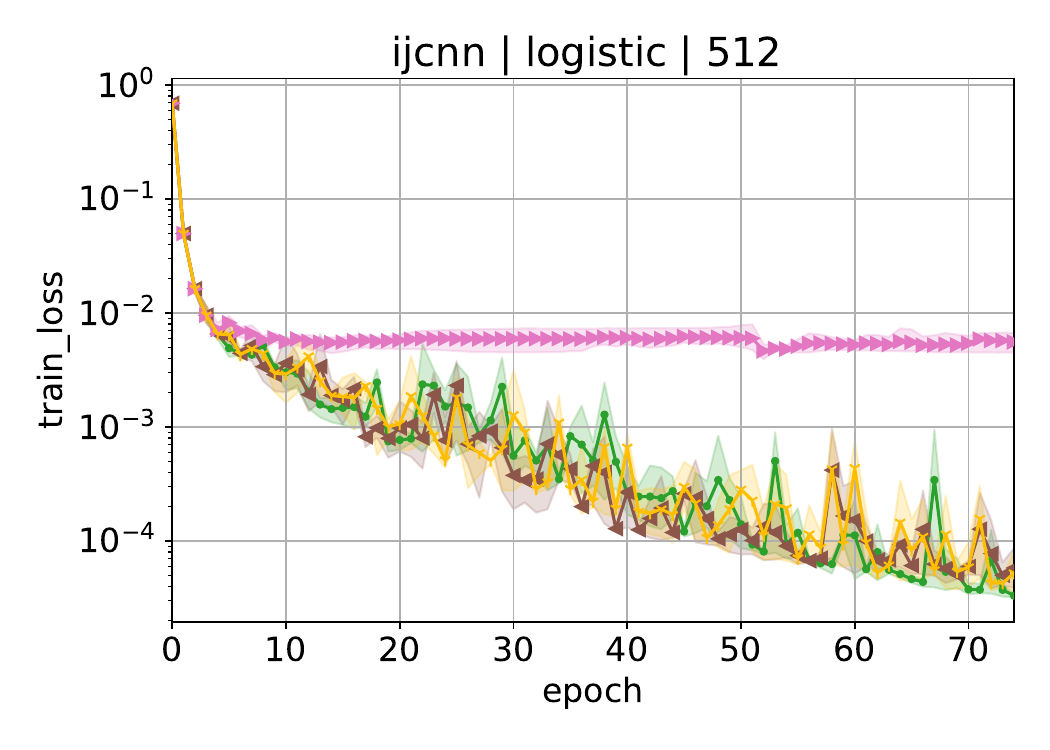}}
	\\
	\subfloat{\includegraphics[width=.32\textwidth]{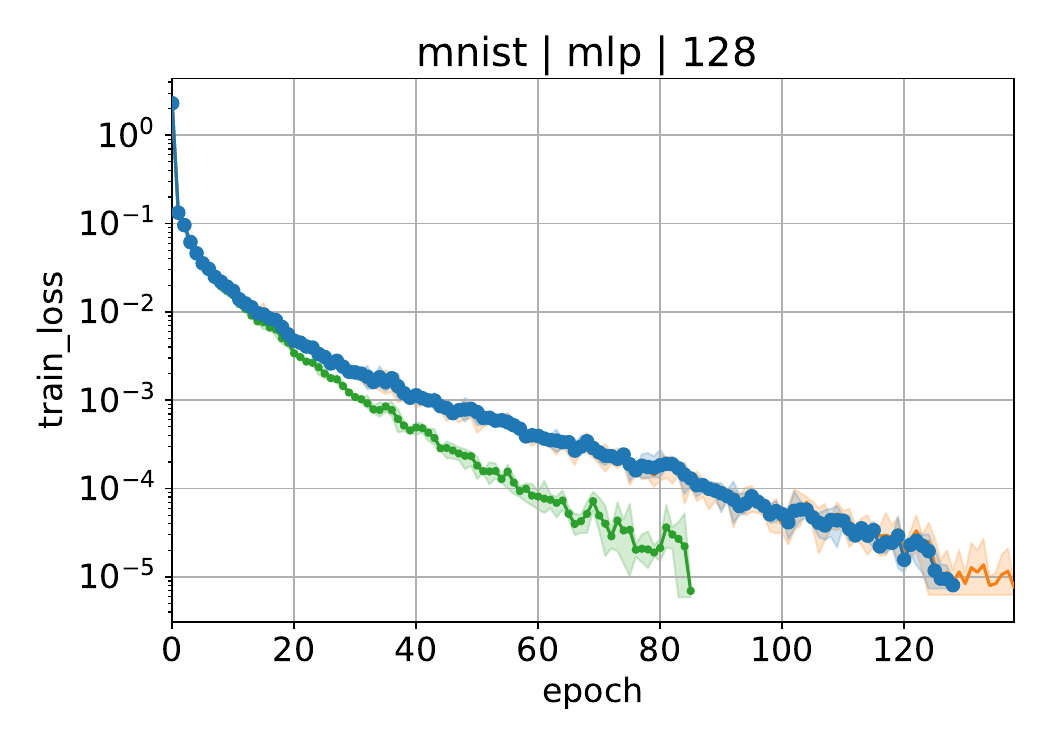}}
	\hfill
	\subfloat{\includegraphics[width=.32\textwidth]{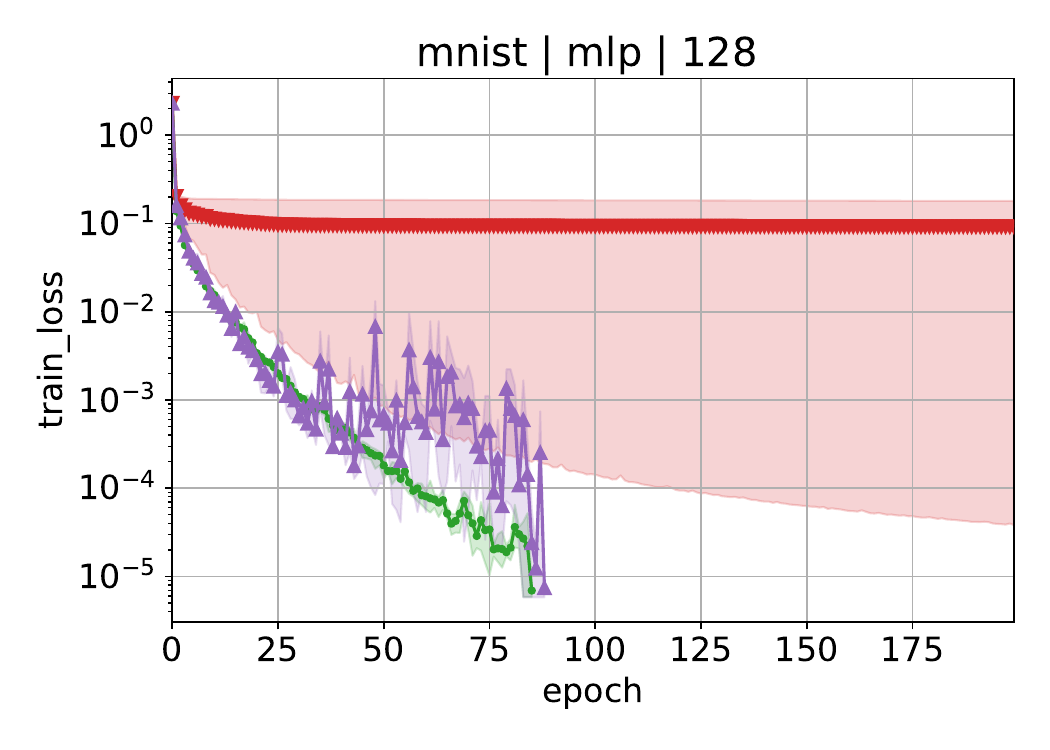}}
	\hfill 
	\subfloat{\includegraphics[width=.32\textwidth]{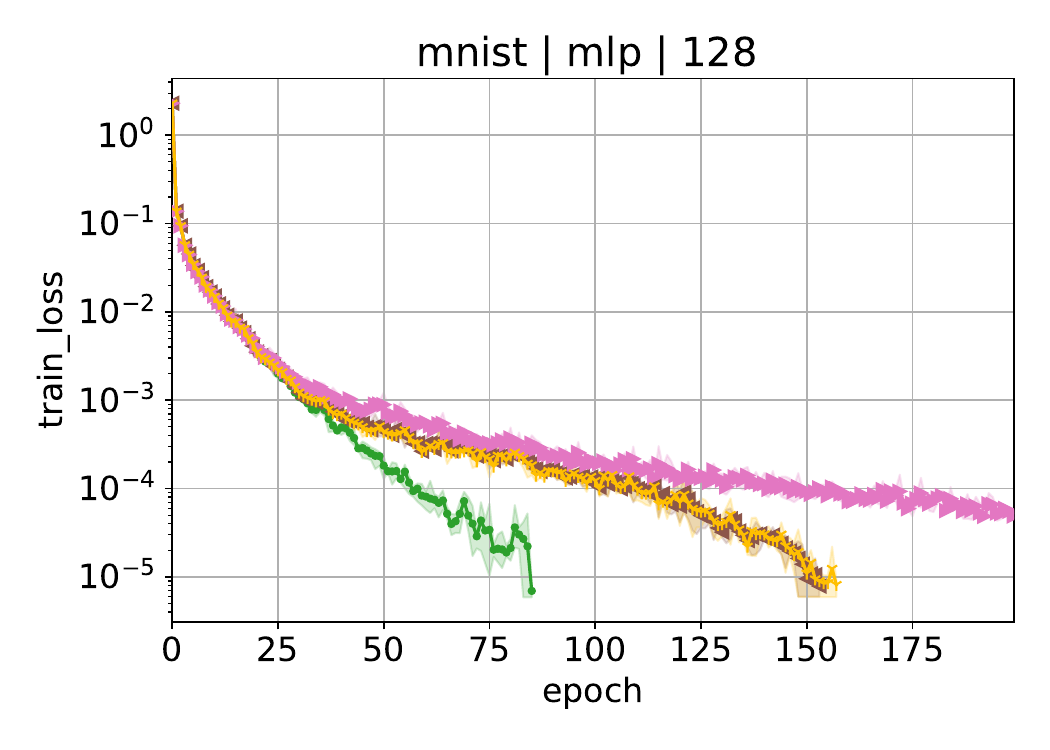}}
	\\
	\setcounter{subfigure}{0}
	\subfloat[Comparison of CG rules. \label{fig:CG_rule_comparison}]{\includegraphics[width=.32\textwidth]{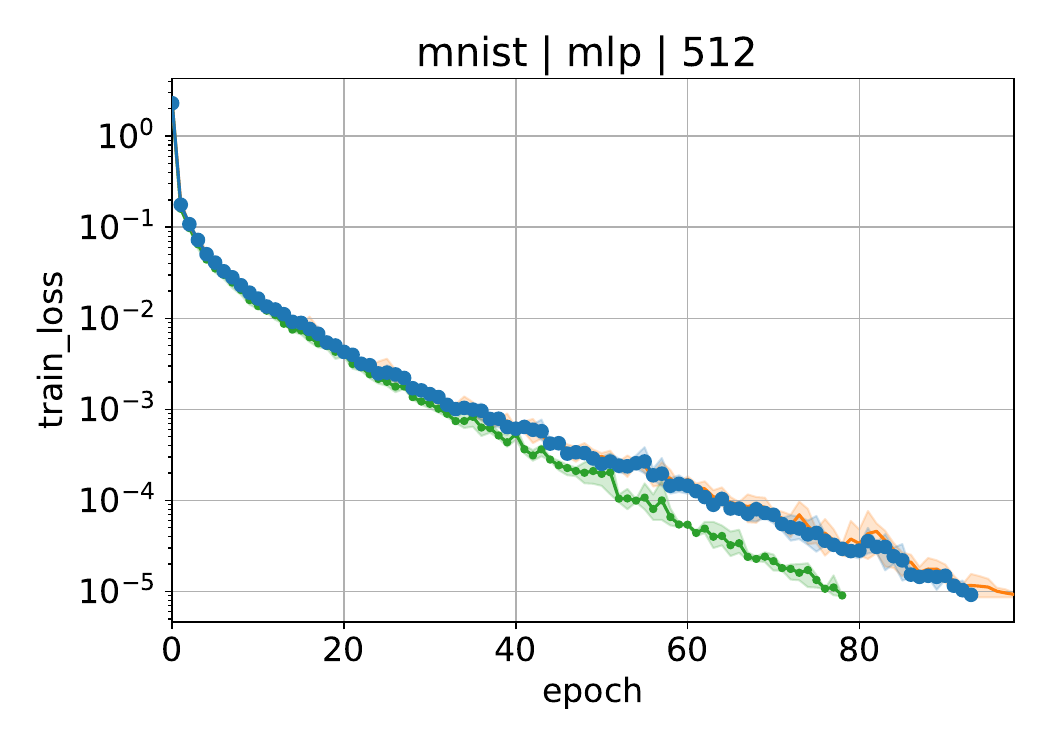}}
	\hfill
	\subfloat[Comparison of tentative step sizes $\alpha_0^k$. \label{fig:CG_tentative_step_comparison}]{\includegraphics[width=.32\textwidth]{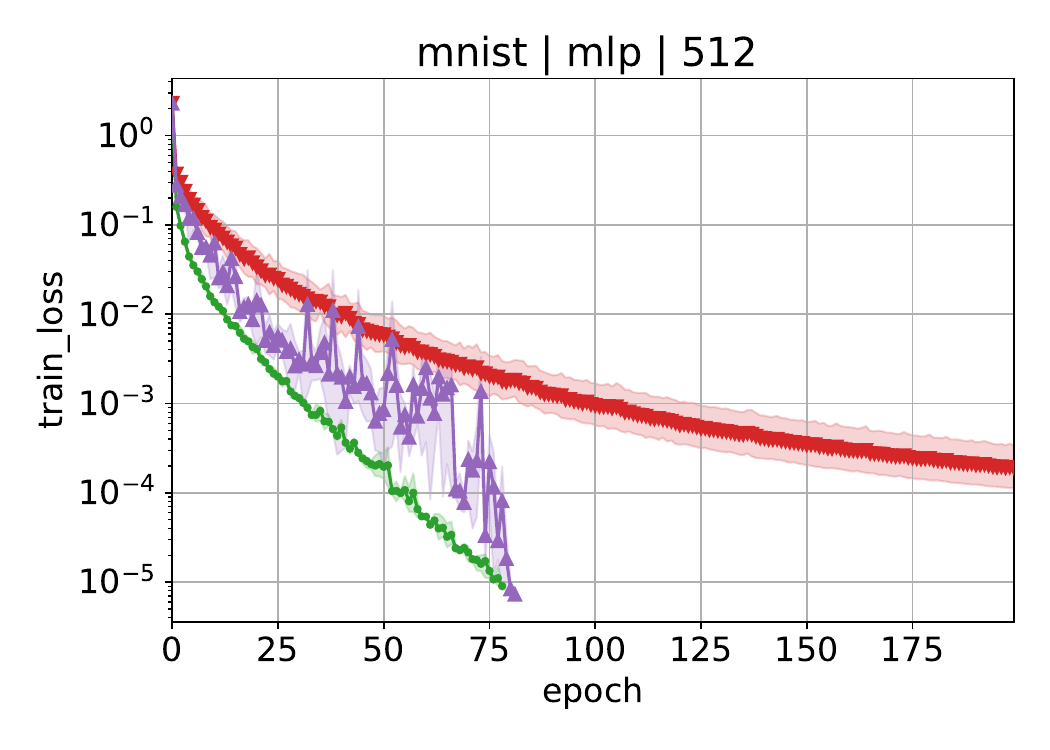}}
	\hfill
	\subfloat[Comparison of the recovery strategies for $d_k$. \label{fig:CG_recover_direction}]{\includegraphics[width=.32\textwidth]{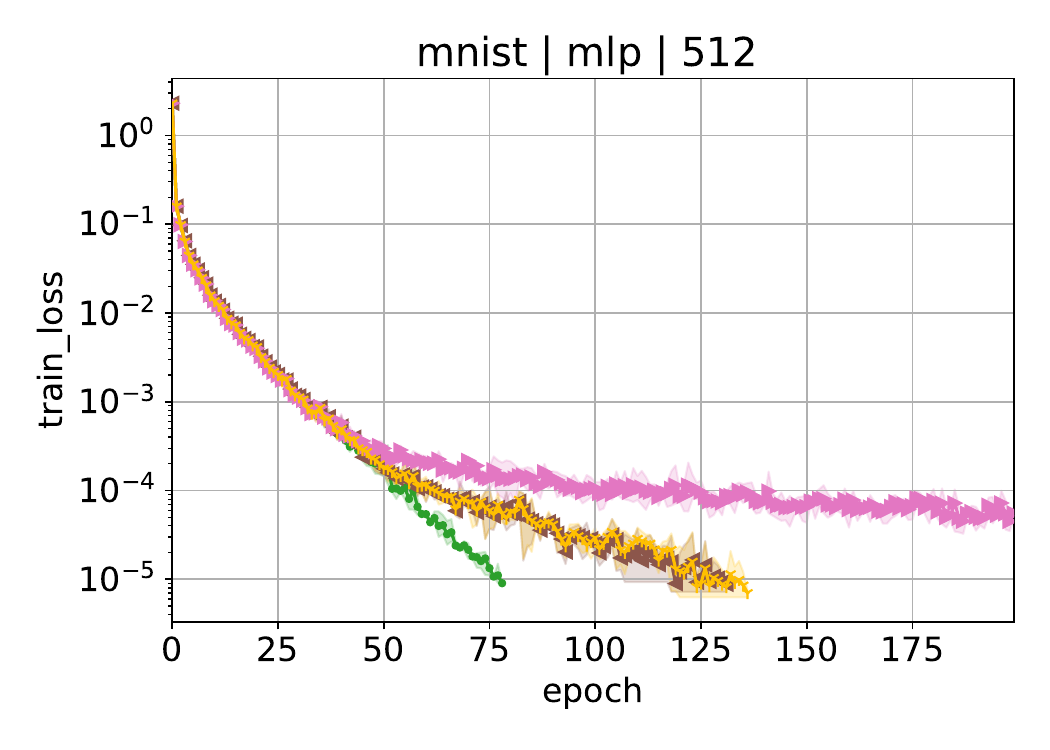}}
	
	\caption{Behavior of MBCG-DP when employed to train a kernel classifier on \texttt{ijcnn} and a multi-layer perceptron on \texttt{MNIST} dataset. The left column refers to the comparison of the FR, HS and PPR rules for selecting $\beta_k$; the central column refers to the comparison of rules for selecting $\alpha_0^k$ (constant, heuristic, SPS); the right column refers to different safeguard strategies for ensuring a descent direction is used (clipping, gradient, inversion, subspace).}
	
	\label{fig:CG_comparison}
\end{figure}

In Figure \ref{fig:CG_tentative_step_comparison} we show the behavior of MBCG-DP when using different approaches for selecting the tentative step $\alpha_0^k$ for the line search: a constant value, the heuristic employed in \cite{vaswani2019, fan2023} and the generalized SPS \eqref{eq:sps_momentum} with $c=1$ were tested. The Fletcher-Reeves rule is used in this experiment; in case $d_k$ is not a descent direction the repeated momentum clipping approach is used, again with contraction factor of $0.5$. The generalized SPS proved to be the most effective choice.

Finally, in Figure \ref{fig:CG_recover_direction} we observe the impact of different recovery options to retrieve a descent direction, namely: switching to the simple (negative) stochastic gradient direction, recomputing the direction according to the subspace optimization strategy proposed in \cite{lapucci2024globally}, inverting the direction and taking $-d_k$ and adopting the repeated momentum clipping approach (contraction factor of $0.5$) from \cite{fan2023}. In this experiment, the Fletcher-Reeves rule is employed in combination with the generalized SPS initial step size. The clipping strategy resulted the preferable choice especially in nonconvex problems. 

This detailed comparison allows us to identify a promising setup to run MBCG-DP in further experiments to compare it to state-of-the-art methods. In what follows, we thus employed the Fletcher-Reeves rule for $\beta_k$, the generalized stochastic Polyak step size \eqref{eq:sps_momentum} for choosing $\alpha_0^k$ and the clipping strategy on $\beta_k$ as a recovery option for obtaining a descent direction. The resulting method will be denoted by the acronym {MBCG\_FR}.

The other algorithms included in the following experiments are: vanilla SGD with momentum, Adam \cite{kingma2015}, SLS \cite{vaswani2019}, PoNoS \cite{galli2023} and MSL\_SGDM \cite{fan2023}. The benchmark of test problems include three convex instances, consisting in the training of an RBF-kernel classifier on \texttt{mushrooms}, \texttt{ijcnn} and \texttt{rcv1} datasets\footnote{datasets available at \href{https://www.csie.ntu.edu.tw/~cjlin/libsvmtools/datasets/binary.html}{csie.ntu.edu.tw/~cjlin/libsvmtools/datasets}}. Moreover, three nonconvex learning tasks are taken into account: once again a multi-layer perceptron with 1000 hidden units and ReLU activations to be trained on \texttt{MNIST}, a three-layered convolutional architecture with 32 channels each, followed by a fully-connected layer of 512 units to be trained on \texttt{FashionMNIST} \cite{xiao2017fashionMNIST} dataset\footnote{available at \href{https://github.com/zalandoresearch/fashion-mnist}{github.com/zalandoresearch/fashion-mnist}} and a ResNet18 architecture to be trained on \texttt{CIFAR10} \cite{krizhevsky2009learning} dataset\footnote{available at \href{https://www.cs.toronto.edu/~kriz/cifar.html}{cs.toronto.edu/~kriz/cifar.html}}. For the latter network, we employed the normalization-free setup proposed in \cite{civitelli2023}.
Batch sizes of 128 and 512 were again considered for all problems.

\begin{table}[htbp]
	\caption{Mini-batch overlap settings leading to the best-performing algorithms across problems.}
	\label{tab:overlap_useful}
	\centering
	\footnotesize
	\begin{tabular}{l|cc|cc|cc|cc|cc|cc|}
		\toprule
		& \multicolumn{6}{c|}{\texttt{RBF-kernel classifier}} & \multicolumn{2}{c|}{\texttt{MLP}} & \multicolumn{2}{c|}{\texttt{CNN}} & \multicolumn{2}{c|}{\texttt{ResNet18}}\\
		& \multicolumn{2}{c}{\texttt{mushrooms}} & \multicolumn{2}{c}{\texttt{ijcnn}} & \multicolumn{2}{c|}{\texttt{rcv1}} & \multicolumn{2}{c|}{\texttt{MNIST}} & \multicolumn{2}{c|}{\texttt{fashionMNIST}} & \multicolumn{2}{c|}{\texttt{CIFAR10}} \\
		\cmidrule(r){2-13}
		& 128 & 512 & 128 & 512 & 128 & 512 & 128 & 512 & 128 & 512 & 128 & 512 \\
		\midrule
		SGD+M & 75\% & 75\% & 75\% & 75\% & 75\% & 75\%   & 75\% & 75\% & 75\% & 75\% & 0\% & 0\% \\
		Adam & 75\% & 75\% & 75\% & 75\% & 75\% & 75\%    & 50\% & 50\% & 75\% & 75\% & 0\% & 10\% \\
		SLS & 75\% & 75\% & 50\% & 75\% & 75\% & 75\%     & 0\% & 75\% & 0\% & 0\% & 0\% & 0\% \\
		PoNoS & 75\% & 75\% & 75\% & 75\% & 50\% & 75\%   & 25\% & 50\% & 0\% & 0\% & 0\% & 0\% \\
		MSL\_SGDM & 75\% & 75\% & 50\% & 75\% & 0\% & 25\% & 10\% & 10\% & 75\% & 0\% & 25\% & 0\% \\
		\bottomrule
	\end{tabular}
\end{table}

\begin{figure}[htbp]
	\centering
	
	\subfloat{\includegraphics[width=.32\textwidth]{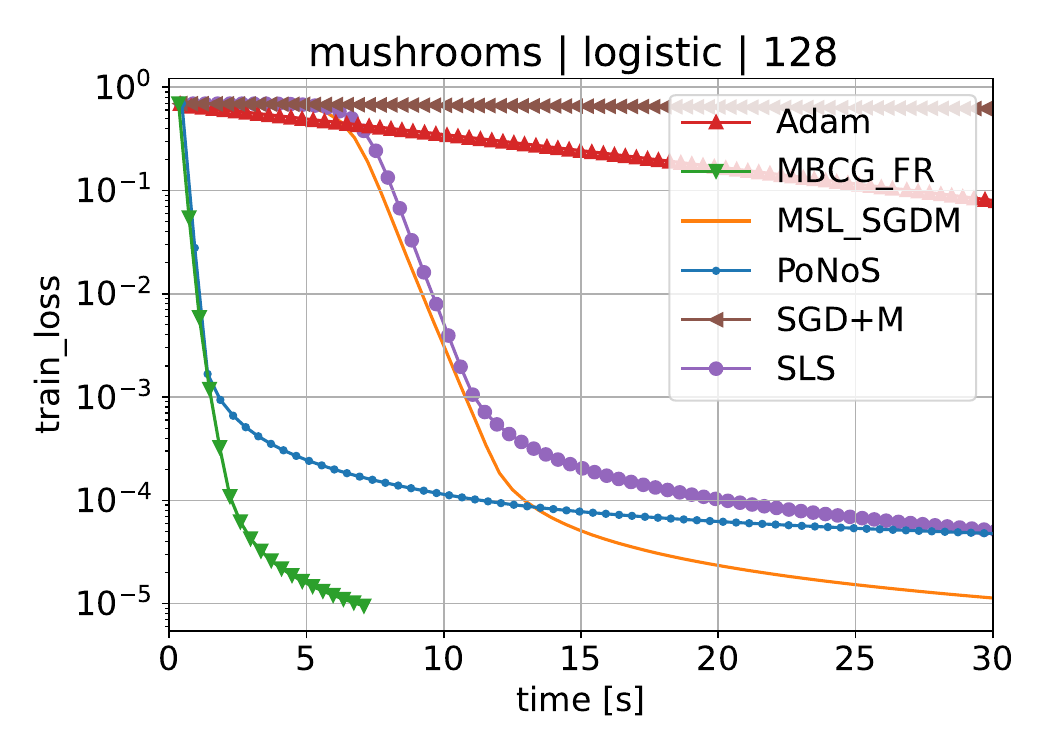}}
	\hfill
	\subfloat{\includegraphics[width=.32\textwidth]{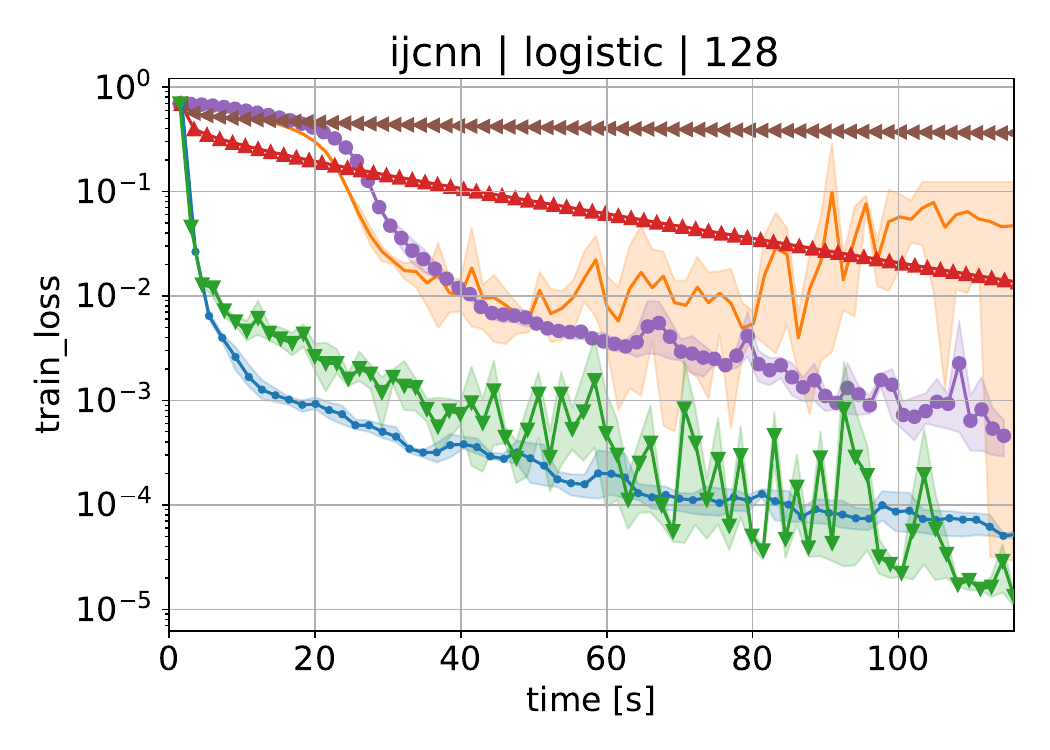}}
	\hfill
	\subfloat{\includegraphics[width=.32\textwidth]{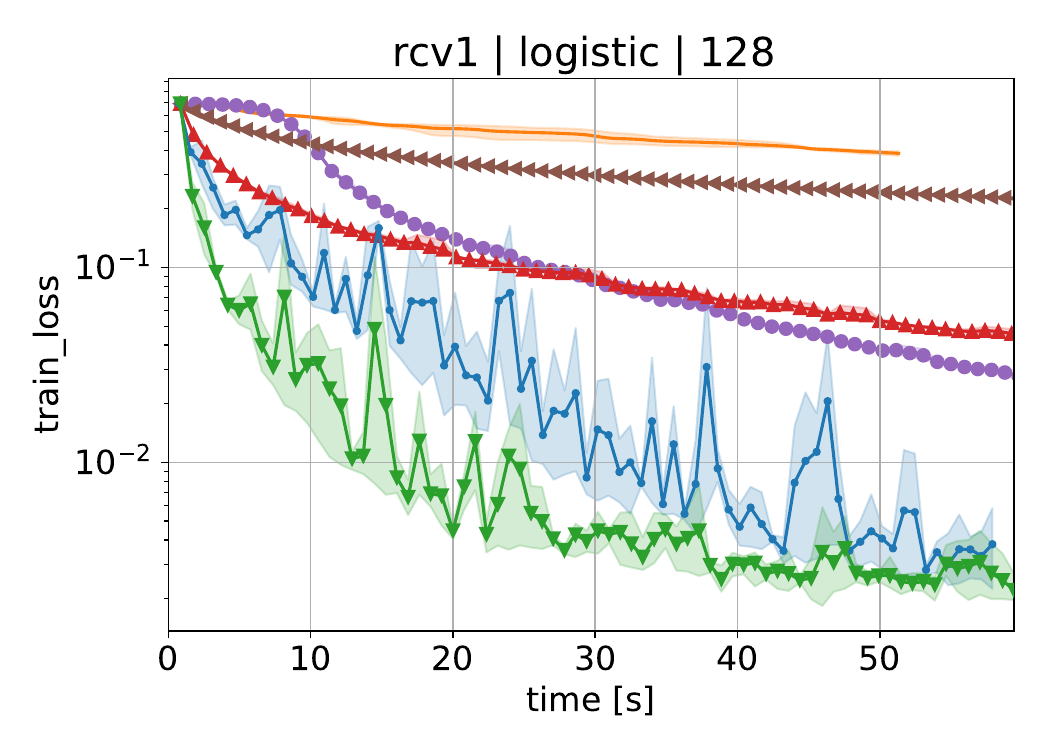}}
	\\
	\setcounter{subfigure}{0}
	\subfloat[RBF-kernel classifier on \texttt{mushrooms} dataset.]{\includegraphics[width=.32\textwidth]{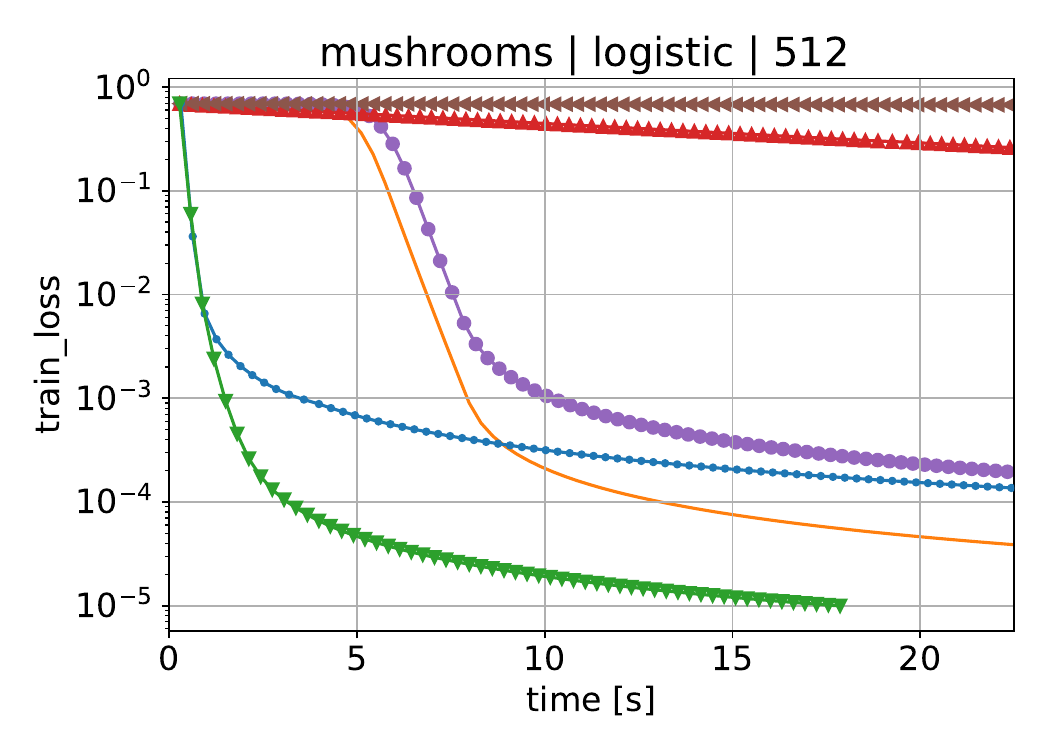}}
	\hfill
	\subfloat[RBF-kernel classifier on \texttt{ijcnn} dataset.]{\includegraphics[width=.32\textwidth]{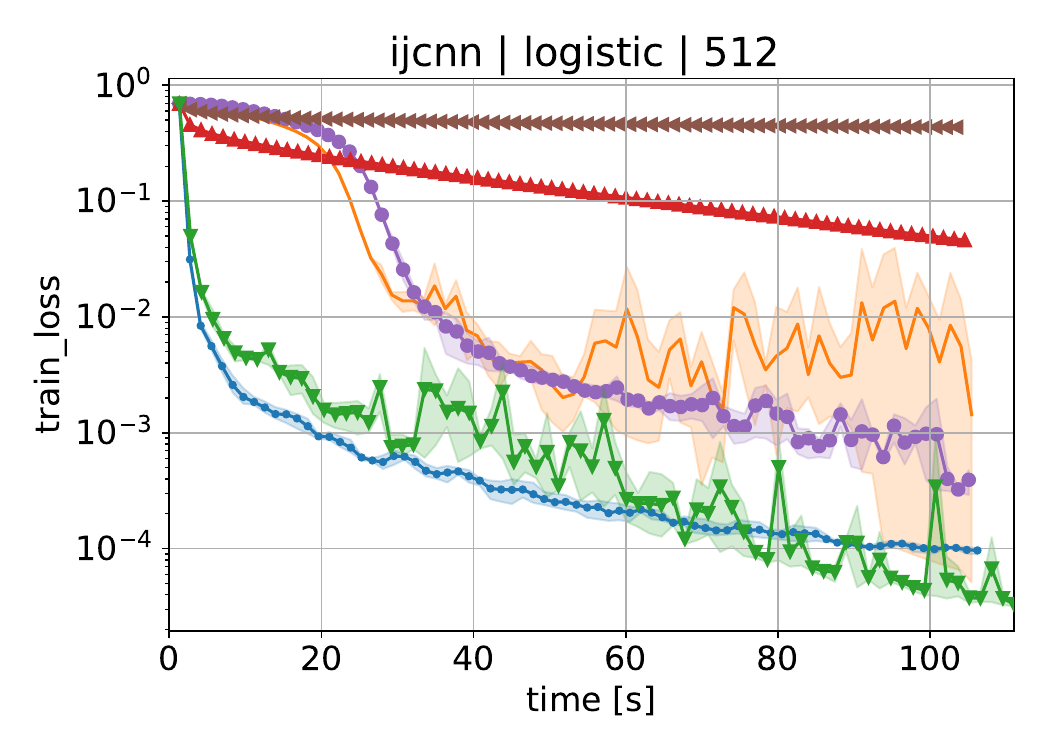}}
	\hfill
	\subfloat[RBF-kernel classifier on \texttt{rcv1} dataset.]{\includegraphics[width=.32\textwidth]{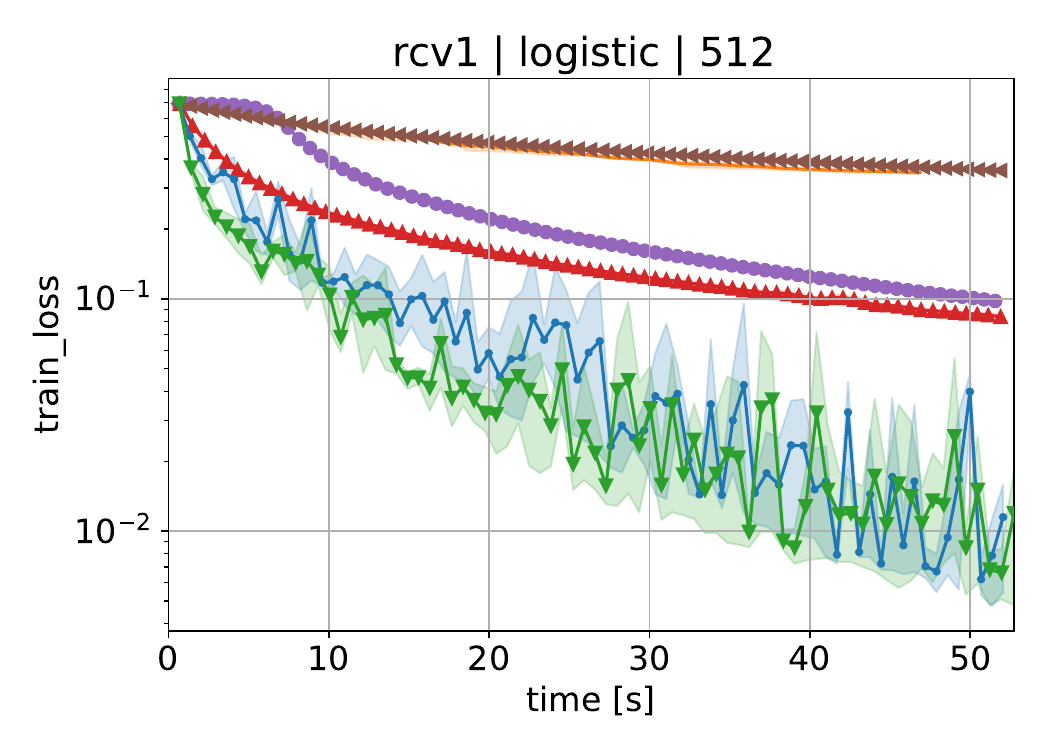}}
	\\
	\subfloat{\includegraphics[width=.32\textwidth]{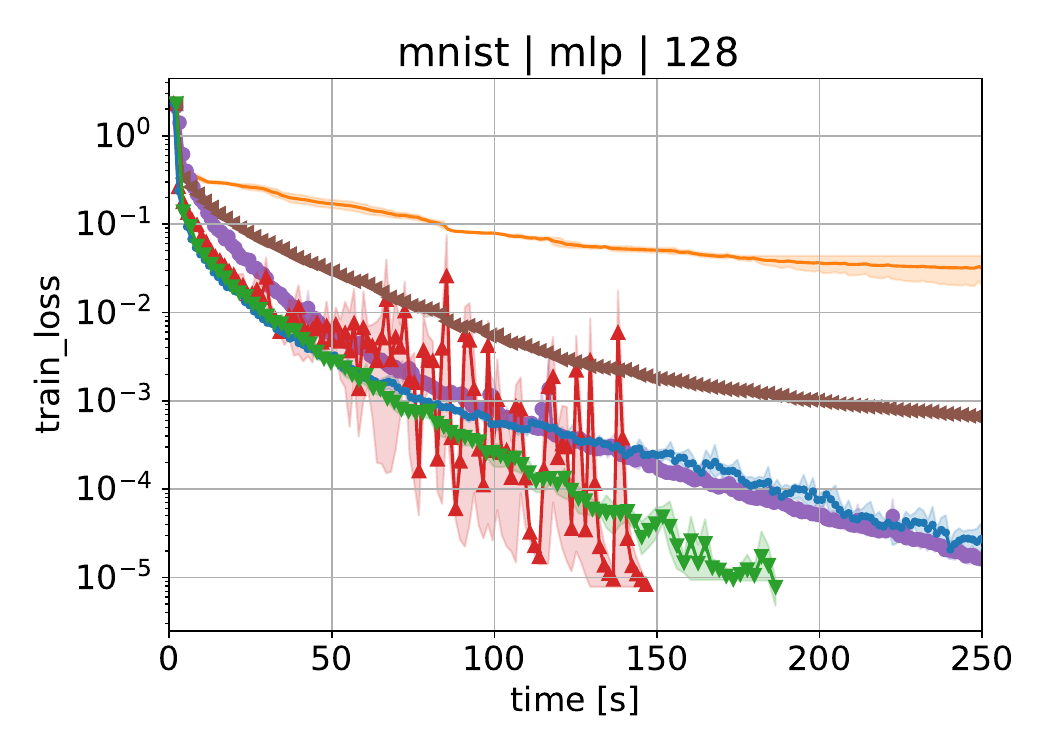}}
	\hfill
	\subfloat{\includegraphics[width=.32\textwidth]{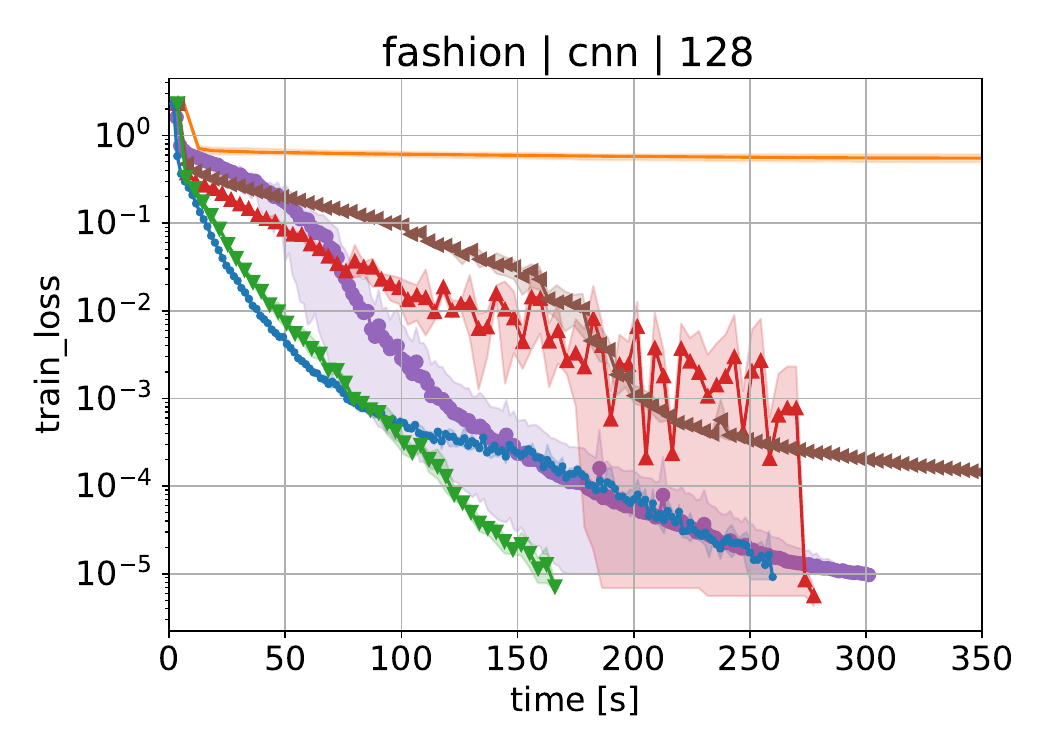}}
	\hfill
	\subfloat{\includegraphics[width=.32\textwidth]{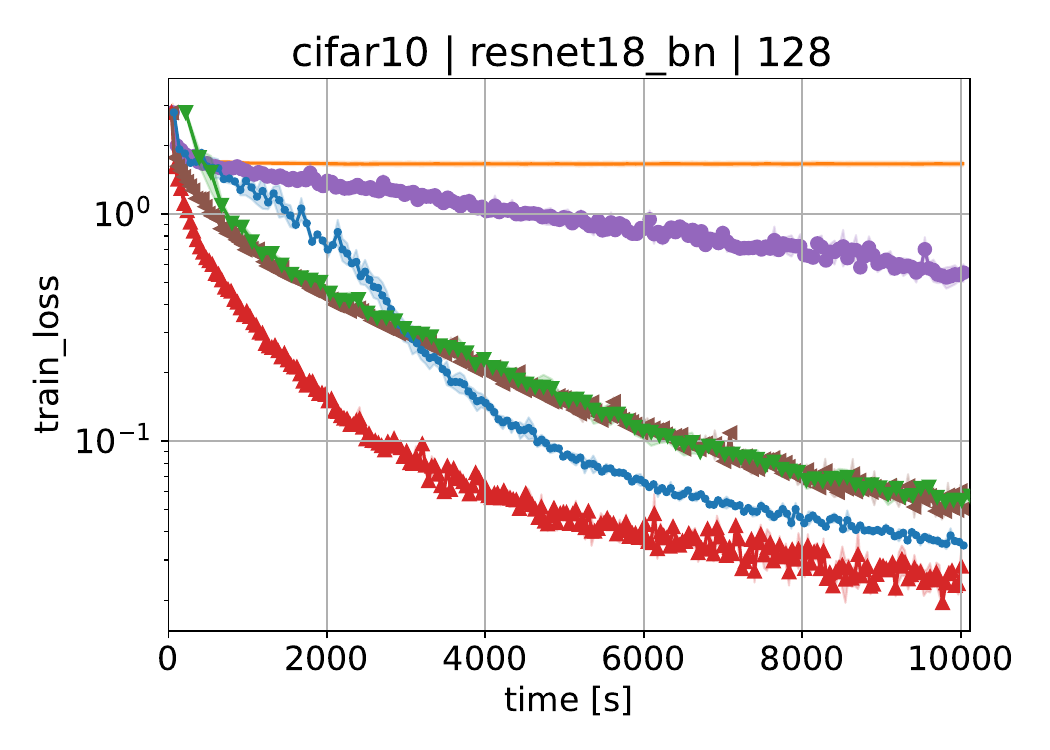}}
	\\
	\setcounter{subfigure}{3}
	\subfloat[Multi-layer perceptron on \texttt{MNIST} dataset.]{\includegraphics[width=.32\textwidth]{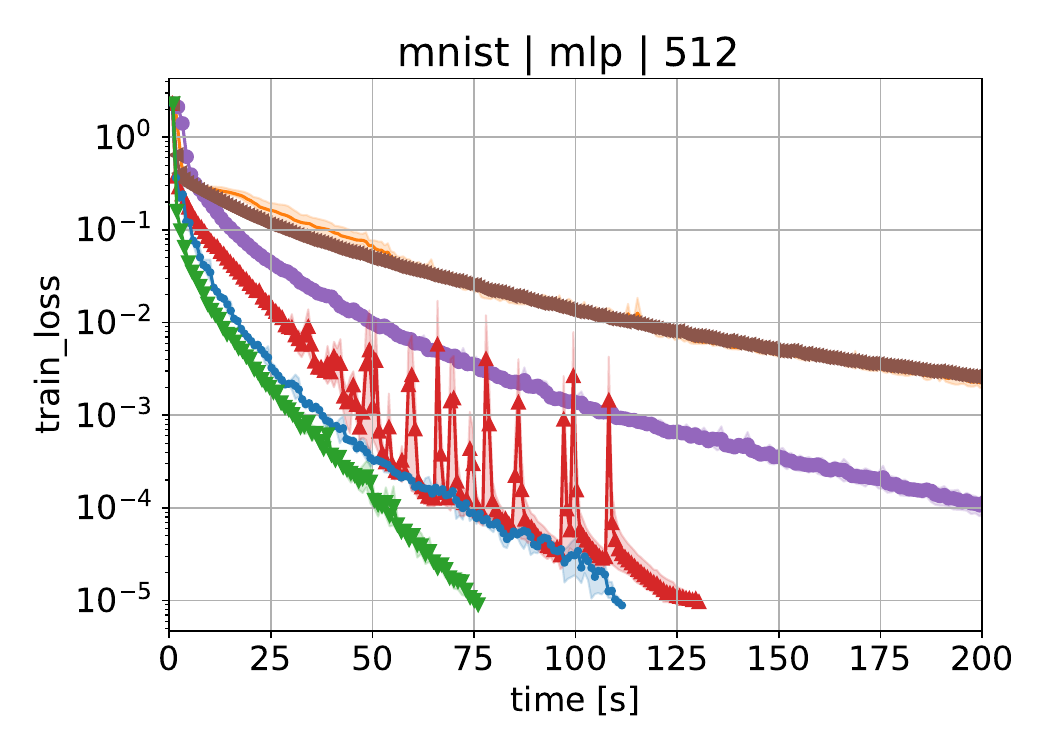}}
	\hfill
	\subfloat[CNN on \texttt{FashionMNIST} dataset.]{\includegraphics[width=.32\textwidth]{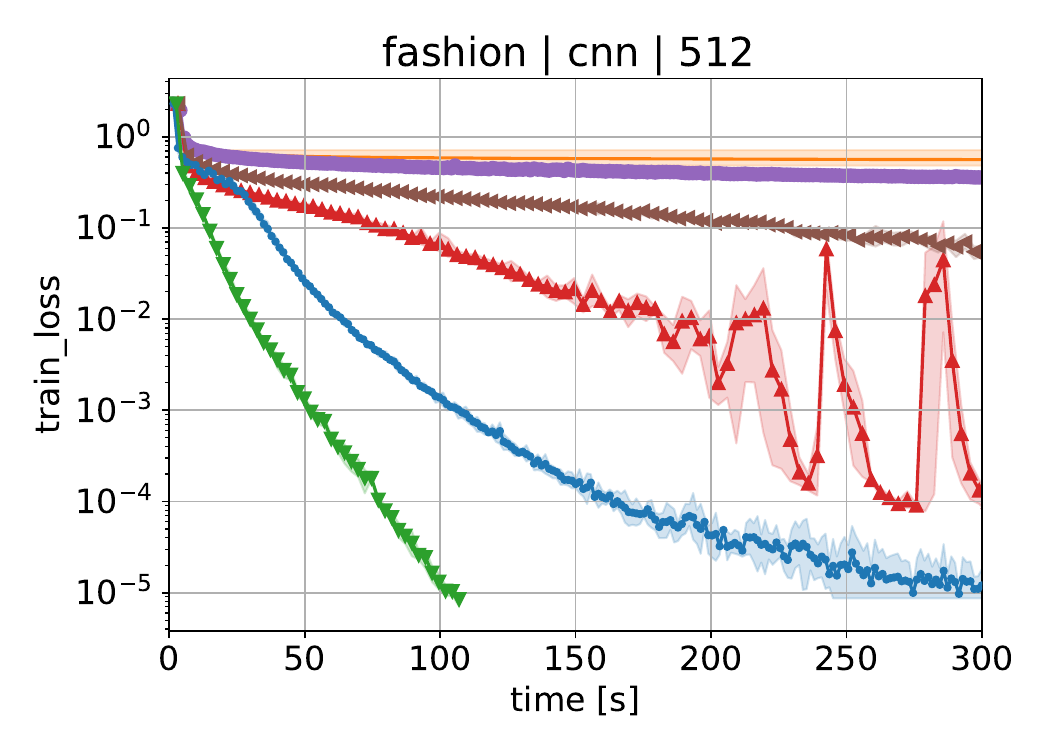}}
	\hfill
	\subfloat[ResNet18 on \texttt{CIFAR10} dataset.]{\includegraphics[width=.32\textwidth]{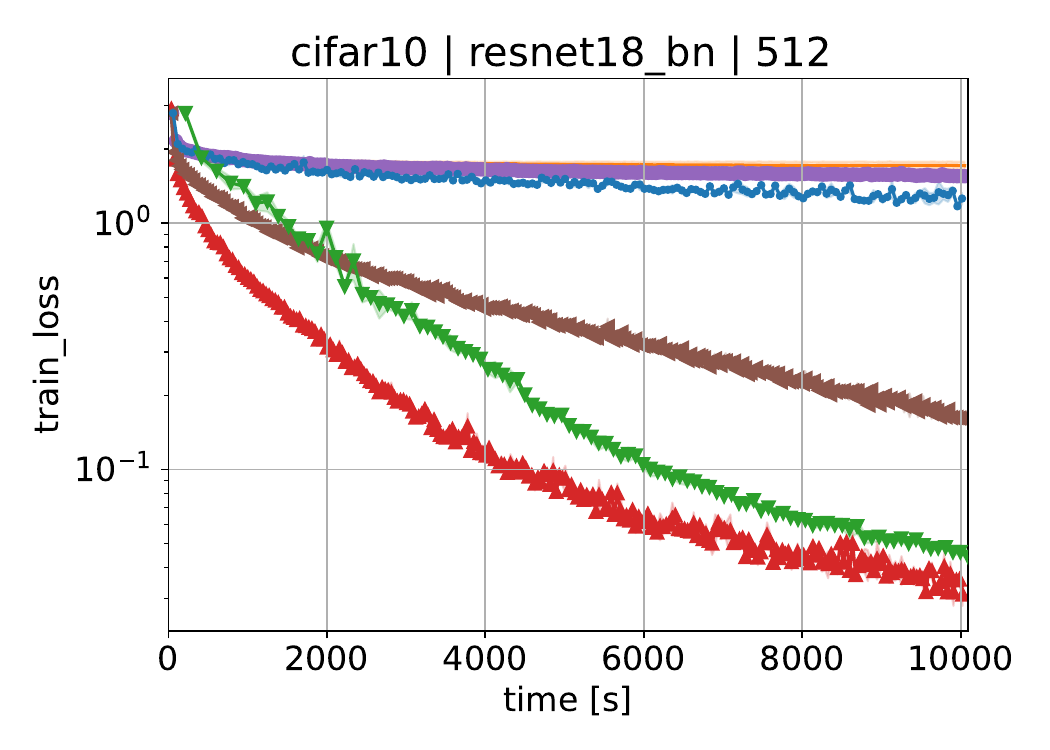}}
	
	\caption{Comparison of the proposed MBCG\_FR algorithm with other algorithms in the literature. The training loss over time is shown for each problem.}
	
	\label{fig:comparison_plots}
\end{figure}

In order to offer a fair comparison between the proposed MBCG-DP method and its competitors, each reference algorithm was evaluated to assess whether employing mini-batch persistency could enhance its performance. For each problem considered, the reference algorithms are tested using an overlap of 0\%, 10\%, 25\%, 50\% and 75\%. An algorithm is deemed to perform better with a certain overlap value if it either terminates before the others with a superior solution or, at the termination time of the first configuration to finish, it holds a better solution. The training loss is compared to asses whether a solution is better than the other. We report in Table \ref{tab:overlap_useful} the outcome of this preliminary analysis.
In convex problems, we observe that overlap is almost always beneficial. However, with larger architectures in nonconvex tasks, this benefit does not appear to hold. We argue that, in such cases, the dominant computational burden shifts from data I/O operations to the evaluation of the training loss and its gradient. Since the number of iterations per epoch with a 50\% overlap is doubled, in this cases the execution time per epoch will also double. In addition, we remark that hyperparameters were selected for each considered method based on preliminary experiments carried out without employing data persistency. A specialized selection of hyperparameters values for each method and overlap percentage might further enhance performance when data persistency is employed.

\begin{figure}[htbp]
	\centering
	
	\subfloat{\includegraphics[width=.32\textwidth]{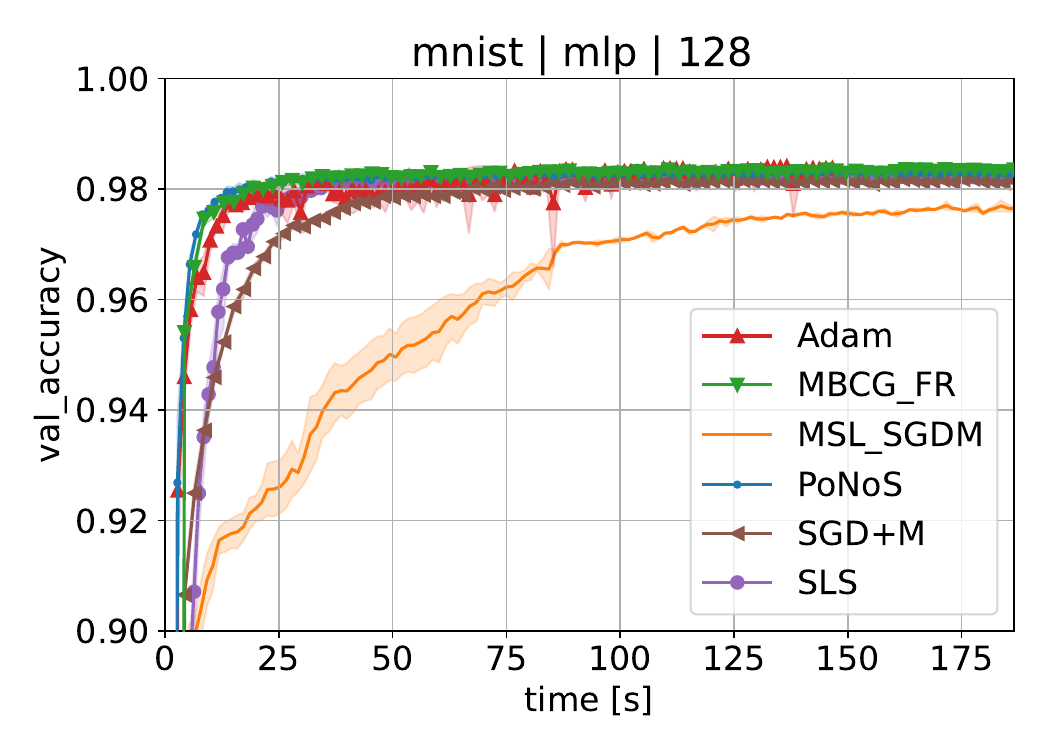}}
	\hfill
	\subfloat{\includegraphics[width=.32\textwidth]{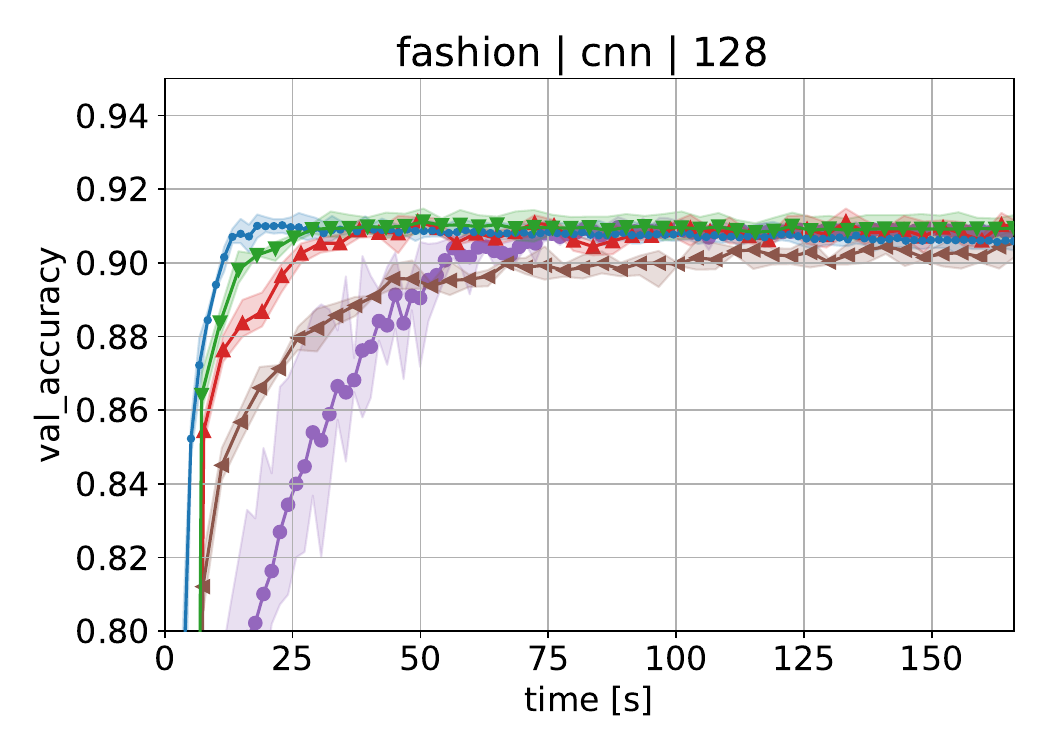}}
	\hfill
	\subfloat{\includegraphics[width=.32\textwidth]{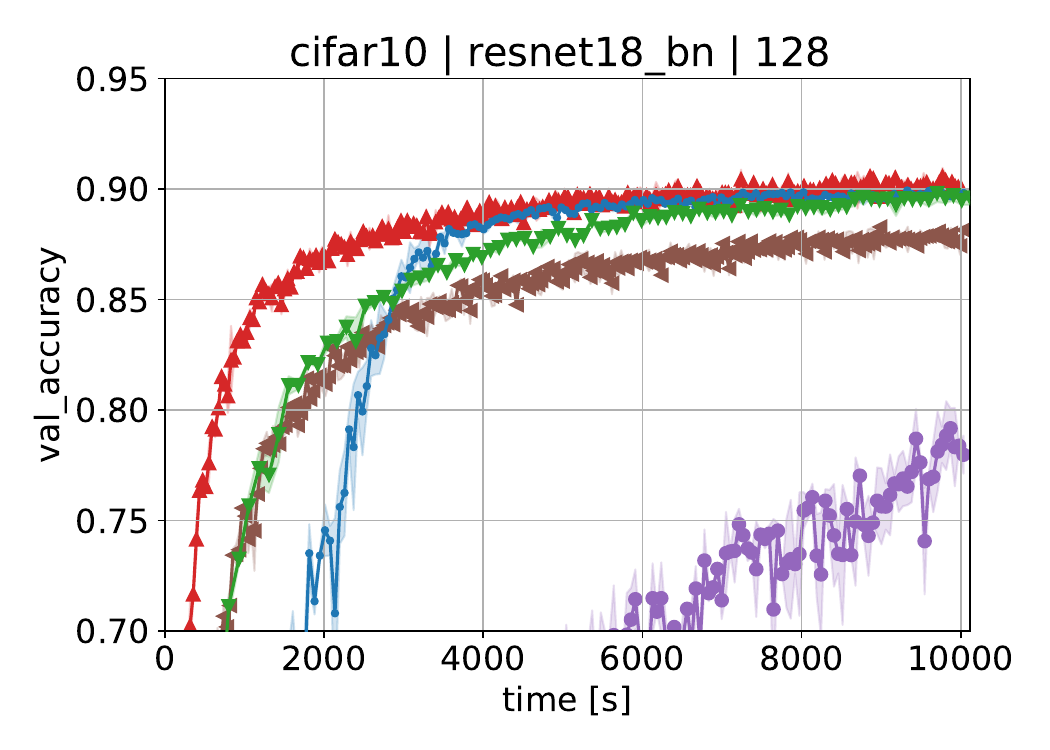}}
	\\
	\setcounter{subfigure}{0}
	\subfloat[Multi-layer perceptron on \texttt{MNIST} dataset.]{\includegraphics[width=.32\textwidth]{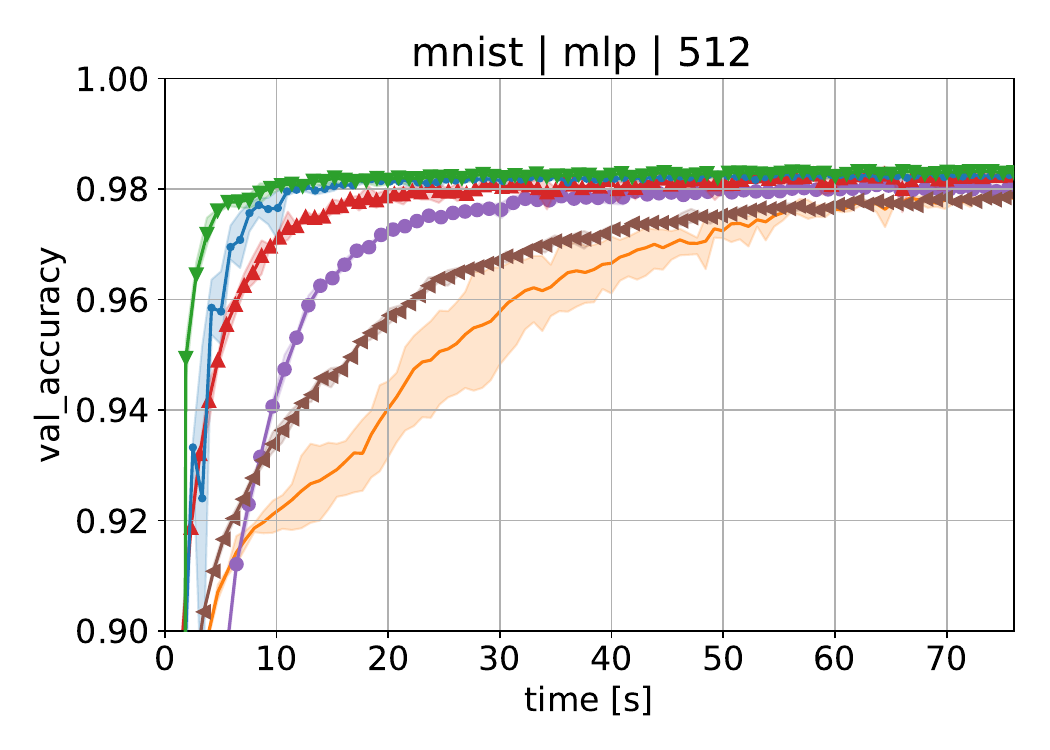}}
	\hfill
	\subfloat[CNN on \texttt{FashionMNIST} dataset.]{\includegraphics[width=.32\textwidth]{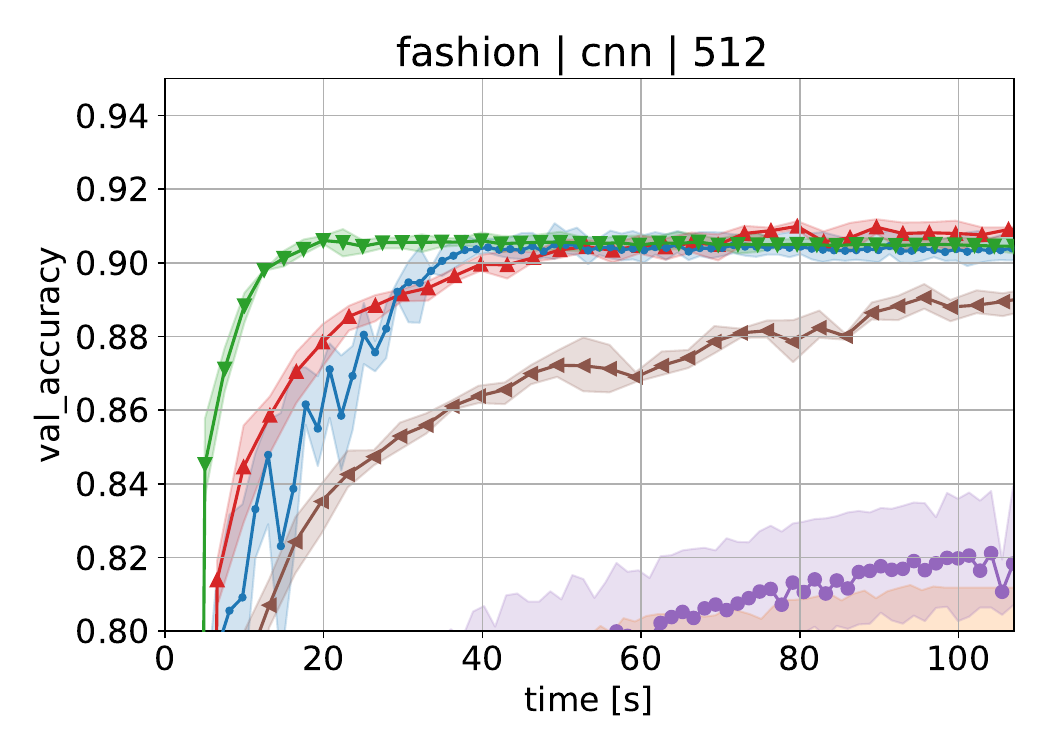}}
	\hfill
	\subfloat[ResNet18 on \texttt{CIFAR10} dataset. \label{fig_acc_cifar10}]{\includegraphics[width=.32\textwidth]{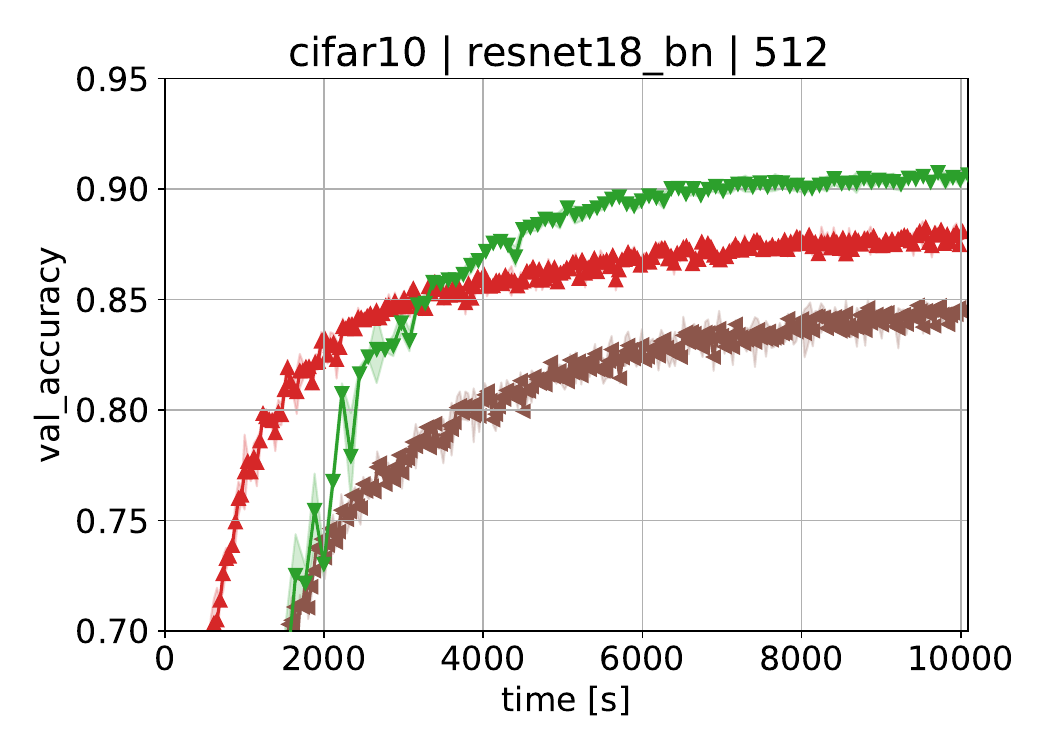}}
	\\
	
	\caption{Comparison of the proposed MBCG\_FR algorithm with other algorithms in the literature. The validation accuracy over time is shown for each problem.}
	
	\label{fig:comparison_accuracies}
\end{figure}

We report in Figure \ref{fig:comparison_plots} and in Figure \ref{fig:comparison_accuracies} the comparison of MBCG\_FR to its competitors in terms of the training loss and the validation accuracy. The algorithms are compared in terms of wall-clock execution time, since employing mini-batch overlap changes the number of updates per epoch, making epoch-based comparison unfair. The averages of 3 runs with different random seeds are reported. 
Standard parameters are employed for the monotone line-search of SLS and MSL\_SGDM, with $\gamma = 0.1$ and $\delta = 0.5$, and also for the nonmonotone line search of PoNoS and MBCG\_FR, with $\gamma=0.5$ and $\delta = 0.5$. For SGD with momentum we use $\alpha = 10^{-2}, \beta=0.9$ for convex problems and $\alpha = 10^{-3}, \beta=0.9$ for nonconvex ones. For Adam we use $\alpha=10^{-3}$ for convex problems and $\alpha=10^{-4}$ for nonconvex ones. For both PoNoS and MBCG\_FR we use $c=1$ in the definition of the generalized SPS, except for the problem of training the ResNet18 architecture on \texttt{CIFAR10} dataset, where $c=0.2$ is employed. Reference algorithms utilize mini-batch persistency only in configurations specified in Table \ref{tab:overlap_useful} with the values there reported, whereas MBCG\_FR always uses a 50\% overlap.
	
As shown in Figure \ref{fig:comparison_plots}, MBCG\_FR is extremely effective at reducing the training loss. In all convex tasks it is capable of obtaining high-quality solutions more quickly than the other algorithms considered. In general, MBCG\_FR is also competitive in nonconvex problems, especially for a batch size of 512. Moreover, looking at trend of the validation accuracy reported in Figure \ref{fig:comparison_accuracies}, we can note that even in cases where Adam is slightly faster, MBCG\_FR is capable of reaching top out-of-sample prediction results. In addition, examining Figure \ref{fig_acc_cifar10}, we can note that the best overall validation accuracy value is reached on the \texttt{CIFAR10} problem by MBCG\_FR with a batch size of 512. 

Based on the results reported in this section, we conclude that the proposed MBCG-DP algorithm is capable of achieving state-of-the-art performance in both convex and nonconvex finite-sum problems, making it particularly suitable for large-scale tasks and machine learning applications. This is especially true when considerable computing resources are available, allowing larger batch sizes to be employed.

\section{Conclusions}
\label{sec:conc}
In this work, we introduced an algorithmic framework based on stochastic line searches and momentum directions to tackle nonlinear finite-sum optimization problems in deep learning scenarios. In order to truly benefit from the addition of momentum terms in the search direction, the method leverages the concept of mini-batch persistency. The resulting approach, for which we state formal convergence results under suitable assumptions, proves to be computationally efficient and effective at solving the class of problems under consideration. In addition, our empirical analysis shows that, for simpler models or small architectures, any algorithm could benefit from the usage of mini-batch persistency.

There are two main research directions that shall be followed at this point: from the one hand, a deep investigation of the convergence properties of algorithms that employ mini-batch persistency in absence of bias-correction terms would be of significant impact. On the other hand, further computational experiments on even larger models, such as transformer architectures, would be important to assess the usefulness of the proposed approach in the most modern applications.

\section*{Declarations}

\subsection*{Funding}

No funding was received for conducting this study.

\subsection*{Competing interests}

The authors have no competing interests to declare that are relevant to the content of this article.

\subsection*{Data Availability Statement}
Data sharing is not applicable to this article as no new data were created or analyzed in this study.

\subsection*{Code Availability Statement}
The code developed for the experimental part of this paper is publicly available at \href{https://github.com/dadoPuccio/MB-Conjugate-Gradient-DP}{github.com/dadoPuccio/MB-Conjugate-Gradient-DP}.

\subsection*{Acknowledgments}
The authors are grateful to the Editor and the reviewers of this manuscript for their constructive comments that helped us to significantly improve the quality of this work. We are very thankful to Professors M.\ Sciandrone, S.\ Lucidi and G.\ Liuzzi for the useful discussions in the early stages of this research.

\bibliography{sn-bibliography}

\end{document}